\documentclass{amsart}
\usepackage[english]{babel}
\usepackage{amsmath, amsthm, amssymb}
\usepackage{enumerate}
\usepackage{hyperref}
\usepackage{cleveref}

\usepackage[a4paper, margin=1in]{geometry}

\theoremstyle{plain}
\newtheorem{theorem}{Theorem}[section]
\newtheorem{proposition}[theorem]{Proposition}
\newtheorem{lemma}[theorem]{Lemma}
\newtheorem{corollary}[theorem]{Corollary}
\theoremstyle{definition}
\newtheorem{example}[theorem]{Example}

\newcommand{\C}{\mathbb{C}}

\newcommand{\Q}{\mathbb{Q}}
\newcommand{\Z}{\mathbb{Z}}
\newcommand{\calA}{\mathcal{A}}
\newcommand{\calD}{\mathcal{D}}
\newcommand{\calE}{\mathcal{E}}
\newcommand{\calI}{\mathcal{I}}
\renewcommand{\O}{\mathcal{O}}
\newcommand{\ind}{\calI}
\DeclareMathOperator{\Nm}{N}
\DeclareMathOperator{\Tr}{Tr}

\title{Totally positive elements with $ m $ partitions exist in almost all real quadratic fields}
\author{Mikul\'{a}\v{s} Zindulka}
\address{Charles University, Faculty of Mathematics and Physics, Department of Algebra,
Sokolovsk\'{a} 83, 186 75 Praha 8, Czech Republic}
\email{zinmik2@gmail.com}

\thanks{The author was supported by Czech Science Foundation (GA\v{C}R) grant 21-00420M, Charles University programme PRIMUS/24/SCI/010, and Charles University project GAUK No. 134824}

\begin{document}

\begin{abstract}
In this paper, we study partitions of totally positive integral elements $ \alpha $ in a real quadratic field $ K $. We prove that for a fixed integer $ m \geq 1 $, an element with $ m $ partition exists in almost all $ K $. We also obtain an upper bound for the norm of $\alpha$ that can be represented as a sum of indecomposables in at most $m$ ways, completely characterize the $\alpha$'s represented in exactly $2$ ways, and subsequently apply this result to complete the search for fields containing an element with $ m $ partitions for $ 1 \leq m \leq 7 $.
\end{abstract}

\maketitle

\section{Introduction}
\label{sec:Intro}

Additive number theory studies properties of subsets of the positive rational integers $ \Z_{\geq 1} $ with respect to addition. Many of its problems can be translated to other commutative semigroups but the results are scarce. One flourishing area at the intersection of additive number theory and combinatorics is the theory of partitions. A partition is a way of representing $ n \in \Z_{\geq 1} $ as a sum of positive integers, while two partitions that differ only by the order of their parts are considered to be the same. There are many results about the partition function $ p(n) $, defined as the number of partitions of $ n $. Hardy and Ramanujan \cite{HR} proved an asymptotic formula for $ p(n) $, which was further improved by Rademacher \cite{Rad37}. Ramanujan \cite{Ram19,Ram21} also discovered his famous congruences for $ p(n) $ modulo powers of $ 5 $, $ 7 $, and $ 11 $. The effort to find congruences for other moduli culminated in the paper \cite{Ono00} by Ono, who proved among other things that for every prime $ m \geq 5 $, there exist infinitely many $ n \in \Z_{\geq 1} $ such that $ p(n) \equiv 0 \pmod{m} $. There are countless partition identities, for example Euler's identity, the first and second Rogers--Ramanujan identities, and the two G\"{o}lnitz--Gordon identities. For an introduction to the theory of partitions, we refer the reader to the elementary \cite{AE} or the more advanced classic \cite{And76}.

Partitions can be made sense of in the setting of number fields. Let $ K $ be a totally real number field, $ \O_K $ its ring of integers, and $ \O_K^+ $ the set of totally positive integers in $ K $. A \emph{partition} of $ \alpha \in \O_K^+ $ is a way of representing $ \alpha $ as a sum of totally positive integers, i.e.,
\[
	\alpha = \lambda_1+\lambda_2+\dots+\lambda_{\ell},
\]
where $ \lambda_i \in \O_K^+ $ for $ 1 \leq i \leq \ell $. Again, two partitions are considered to be the same if they differ only by the order of their parts. We define $ p_K(\alpha) $ to be the number of partitions of $ \alpha $ and call $ p_K $ the \emph{partition function} associated with the number field $ K $. We also set $ p_K(0) := 1 $.

Compared with the integer partition function $ p(n) $, the function $ p_K(n) $ is poorly understood. The problem to determine the asymptotic behavior of $ p_K(\alpha) $ as the norm of $ \alpha $ grows to infinity was proposed by Rademacher \cite{Rad50} and solved by Meinardus, first for a real quadratic field \cite{Mei53} and then in general \cite{Mei56}. If $ K $ is a totally real field of degree $ d $ and discriminant $ \Delta $, then \cite[Satz 3, p. 346]{Mei56} shows that
\[
	\log p_K(\alpha) = (d+1)\left(\frac{\zeta(d+1)}{\sqrt{\Delta}}\Nm(\alpha)\right)^{\frac{1}{d+1}}\left(1+o(1)\right)
\]
as $ \Nm(\alpha) \to \infty $. Here $ \zeta $ denotes the Riemann zeta function. This result was further generalized to a number field which is not required to be totally real by Mitsui \cite{Mit78}.

Some basic properties of $ p_K $ were established in a paper of Stern and the author \cite{SZ}. In particular, $ p_K(\alpha) $ satisfies a recurrence formula similar to a well known recurrence for $ p(n) $ \cite[Theorem 1]{SZ}, which can be used to compute particular values of $ p_K(\alpha) $. It can also be applied to prove a result about the parity of $ p_K(n) $, where $ n \in \Z_{\geq 1} $ \cite[Theorem 2]{SZ}. In two recent papers, Jang, Kim, and Kim developed a framework for extending partition identities from $ \Z $ to $ \O_K $ and used it to prove a version of the Euler--Glaisher Theorem \cite[Theorem 4.1]{JKK1}, Sylvester's Theorem \cite[Theorem 3.2]{JKK2}, and the Rogers--Ramanujan identities over a totally real field \cite[Theorem 3.8, Corollary 3.13]{JKK2}. The Frobenius problem for totally positive integers was studied by Fukshansky and Shi \cite{FS}.

A natural generalization of integer partitions is \emph{vector partitions}; see \cite[Chapter 12]{And76} for an introduction to this topic. Vector partitions of $ \mathbf{v} \in \Z^d $ are of the form $ \mathbf{v} = \mathbf{v}_1+\mathbf{v}_2+\dots+\mathbf{v}_{\ell} $, where $ \mathbf{v}_i \in \Z_{\geq 0}^d $ is non-zero for $ 1 \leq i \leq \ell $. If $ K $ is of degree $ d $, we can fix an integral basis $ (\alpha_1, \alpha_2, \dots, \alpha_d) $ and view $ \alpha \in \O_K $ as a vector in $ \Z^d $. Thus, the two types of partitions are closely related, even though the positivity condition is different: we require the parts to lie in the ``totally positive cone'' defined by
\[
	\left\{(x_1,x_2,\dots,x_d) \in \Z^d\,\vert\,\alpha := x_1\alpha_1+x_2\alpha_2+\dots+x_d\alpha_d \succ 0\right\}.
\]
The proofs of some of the most fascinating results in the theory of integer partitions such as the ones in the previously mentioned paper \cite{Ono00} use modular forms. A major obstacle to extending these results to either vector partitions or partitions in number fields is that the corresponding generating functions are not known to satisfy any interesting modular properties.

Let us mention that there is also another type of partitions in number fields. Namely, for a fixed $ \beta \in \C $, one can consider partitions of $ \alpha \in \C $, where the parts are powers of $ \beta $, i.e., partitions of the form
\[
	\alpha = a_j\beta^j+a_{j-1}\beta^{j-1}+\dots+a_1\beta+a_0,\qquad a_i \in \Z_{\geq 0}.
\]
If $ \beta = m \in \Z_{\geq 2} $, then these are the so-called \emph{$ m $-ary partitions}. Kala and the author \cite[Theorem 1]{KZ} characterized the quadratic irrationals $ \beta $ such that the number of partitions of $ \alpha $ is finite for every $ \alpha \in \C $. The theorem was extended to an algebraic number $ \beta $ of arbitrary degree by Dubickas \cite[Theorem 1.3]{Du24}.

Stern and the author \cite{SZ} studied the following problem: For a fixed $ m \in \Z_{\geq 1} $, determine all the real quadratic fields $ K $ such that $ p_K(\alpha) = m $ for some $ \alpha \in \O_K^+ $. We let
\[
	\calD(m) := \left\{D \in \Z_{\geq 2}\text{ squarefree}\ |\ m \notin p_K(\O_K^+)\right\},
\]
i.e., $ \calD(m) $ is the set of $ D $'s such that $ m $ does not belong to the range of $ p_K $ in $ K = \Q(\sqrt{D}) $.

If $ m $ is one of the values of the integer partition function $ p(n) $, then $ \calD(m) $ is finite (this is a consequence of \cite[Theorem 7]{SZ}). In particular, $ \calD(m) $ was determined for $ m \in \{1,2,3,5,7,11\} $ in \cite[Theorem 3]{SZ} and additionally for $ m = 4 $ in \cite[Theorem 11]{SZ}:
\begin{gather*}
	\calD(1) = \emptyset,\quad\calD(2) = \emptyset,\quad\calD(3) = \{5\},\quad\calD(4) = \emptyset,\quad\calD(5) = \{2, 3, 5\},\\
	\calD(7) = \{2, 5\},\quad\calD(11) = \{2, 3, 5, 6, 7, 13, 21\}.
\end{gather*}

The main result of the present paper is that an element with $ m $ partitions exists in almost all real quadratic fields in the sense that real quadratic fields $ K $ having an element $ \alpha $ with $ p_K(\alpha) = m $ are of density $ 1 $. This is surprising in view of the fact that by the asymptotic formula of Meinardus, $ p_K(\alpha) $ grows exponentially with the cube root of the norm of $ \alpha $ in a fixed $ K $. In particular, the result of Meinardus implies that there is no totally real $ K $ such that for all $ m \in \Z_{\geq 1}$, there exits an $ \alpha \in \O_K^+ $ with $ p_K(\alpha) = m $.

\begin{theorem}
\label{thm:S}
Let $ m \in \Z $, $ m \geq 4 $, and
\[
	\calE(m, X) := \left\{2 \leq D \leq X\,|\,D\text{ squarefree}, \{1,2,\dots,m\}\not\subset p_K(\O_K^+)\text{ for }K = \Q(\sqrt{D})\right\}.
\]
For every $ X \geq 2 $ satisfying $ X \geq (2m-5)^{12}(\log X)^4 $, we have
\[
	\#\calE(m, X) < 100(2m-5)^{3/2}(\log X)^{3/2}X^{7/8}.
\]
\end{theorem}
\begin{proof}
This will be proved as a part of \Cref{thm:S2}.
\end{proof}

A consequence of \Cref{thm:S} is that $ \calD(m) $ has density zero for every $ m \in \Z_{\geq 1} $. This is stated explicitly in \Cref{cor:Dm0}.

The above list of the sets $ \calD(m) $ does not include the case $ m = 6 $. We completely characterize the elements $ \alpha \in \O_K^+ $ such that $ p_K(\alpha) = 6 $, and then determine $ \calD(6) $ in \Cref{thm:E6}.

A distinguished role in the study of partitions in $ \O_K^+ $ is played by the so-called indecomposable elements. An element $ \alpha \in \O_K^+ $ is \emph{indecomposable} if there do not exist $ \beta, \gamma \in \O_K^+ $ such that $ \alpha = \beta+\gamma $. Thus, these are the elements $ \alpha \in \O_K^+ $ satisfying $ p_K(\alpha) = 1 $. They were used by Siegel \cite{Sie45} in 1945 to show that if $ K $ is a number field, $ K \neq \Q $ and $ K \neq \Q(\sqrt{5}) $, then there exists an $ \alpha \in \O_K^+ $ which cannot be represented as a sum of squares, and they found many other applications in the theory of universal quadratic forms since then \cite{BK1, BK2, Kim00}.

Let $ K = \Q(\sqrt{D}) $, where $ D \in \Z_{\geq 2} $ is squarefree. Dress and Scharlau \cite[Theorem 2]{DS} characterized the indecomposables in $ \O_K^+ $ in terms of the continued fraction of $ \sqrt{D} $. As a corollary, they proved the following: the norm of an indecomposable $ \alpha \in \O_K^+ $ satisfies
\[
	\Nm(\alpha) \leq c_D := \begin{cases}
		D&\text{if $ D \equiv 2,3 \pmod{4} $,}\\
		\frac{D-1}{4}&\text{if $ D \equiv 1 \pmod{4} $.}
	\end{cases}
\]
If the fundamental unit $ \varepsilon $ in $ \O_K $ has norm $ -1 $, then the bound is optimal. In the case when the norm of $ \varepsilon $ equals $ 1 $, it was refined in \cite{JK} and \cite{TV}.

In general, no satisfactory description of indecomposables is known in fields of degree higher than $ 2 $. It is not difficult to show that there exists a constant $ c_K > 0 $ such that if $ \alpha \in \O_K^+ $ is indecomposable, then $ \Nm(\alpha) \leq c_K $ (for the argument, see \cite[p. 292]{DS}). Brunotte \cite{Bru83} obtained a bound for the norm in terms of the group of units of $ K $. Kala and Yatsyna \cite[Theorem 5]{KY} found a simple argument showing that in a totally real number field $ K $ of arbitrary degree $ d $ and discriminant $ \Delta $, every indecomposable element $ \alpha \in \O_K^+ $ satisfies $ \Nm(\alpha) \leq \Delta $. This bound is usually much better than Brunotte's. In \cite[Theorem 6]{KY} it is used to construct a universal diagonal quadratic form over $ \O_K $ of rank $ \leq C\cdot\Delta(\log \Delta)^{d-1} $, where the constant $ C>0 $ depends only on~$ d $.

Kala and Tinková \cite[Theorem 1.2]{KT} characterized indecomposables in the family of the simplest cubic fields. \v{C}ech, Lachman, Svoboda, Tinkov\'{a} and Zemkov\'{a} \cite{CL} studied biquadratic fields and found sufficient conditions for an element to be indecomposable. Recently Man \cite{Man24} completely characterized the indecomposables in certain families of biquadratic fields, see for example \cite[Theorem 1.6]{Man24}.

Let us introduce one more piece of notation. We define $ p_K(\alpha|\ind) $ as the number of partitions of $ \alpha $ with indecomposable parts, i.e., partitions $ (\lambda_1,\lambda_2,\dots,\lambda_{\ell}) $ such that $ \lambda_i $ is indecomposable for $ 1 \leq i \leq \ell $. Hejda and Kala \cite{HK} called an element $ \alpha \in \O_K^+ $ \emph{uniquely decomposable} if $ p_K(\alpha|\ind) = 1 $. By \cite[Theorem 10]{HK}, the norm of these elements can bounded~as
\[
	\Nm(\alpha) < \sqrt{\Delta}\left(2\sqrt{\Delta}+1\right)\left(3\sqrt{\Delta}+2\right).
\]
A natural question is how to estimate the norm of a non-uniquely decomposable element. Our next result is an extension of the theorem of Hejda and Kala.

\begin{theorem}
\label{thm:N}
Let $ K = \Q(\sqrt{D}) $, where $ D \in \Z_{\geq 2} $ is squarefree and let $ m \in \Z_{\geq 1} $. If $ \alpha \in \O_K^+ $ can be represented as a sum of indecomposables in at most $ m $ ways, then
\[
	\Nm(\alpha) < m^2(2m+1)(2m+3)\cdot \sqrt{\Delta}\left(\sqrt{\Delta}+2\right)^2,
\]
where $ \Delta $ is the discriminant of $ K $.
\end{theorem}
\begin{proof}
This will be proved as \Cref{thm:N'}.
\end{proof}

The norm of a convergent is bounded by $ 2D^{1/2} $, the norm of an indecomposable by $ D $, and the norm of a uniquely decomposable element by a constant times $ D^{3/2} $. The theorem states that for a fixed $ m $, the norm of an element $ \alpha \in \O_K^+ $ such that $ p_K(\alpha|\ind) \leq m $ is also bounded by a constant times $ D^{3/2} $. Let us compare this with a bound for the norm of an $ \alpha $ such that $ p_K(\alpha) = m $. We will prove in \Cref{thm:Np} that if $ m \geq 2 $, then $ \Nm(\alpha) $ is bounded by a constant (depending only on $ m $) times $ \sqrt{D} $. This is not true for $ m = 1 $, when $ \alpha $ is indecomposable. 

In \Cref{thm:D2}, we characterize all the elements $ \alpha \in \O_K^+ $ which can be represented as a sum of indecomposables in exactly $ 2 $ ways, i.e., $ p_K(\alpha|\ind) = 2 $. We obtain an improved bound for the norm of these elements in \Cref{thm:N2}.

Let us say something about the importance of the problems considered in this paper. Proving that $ p_K $ has certain property for a fixed $ K $ is usually hard. For example, the problem whether there exist infinitely many distinct values $ p_K(\alpha) $ such that $ p_K(\alpha) \equiv 0 \pmod{m} $, where $ m \in \Z_{\geq 2} $, is wide open. Thus, it was of interest to find some weaker property of the range that we could prove. This naturally led to the question of deciding whether a given $ m $ belongs to the range of $ p_K $. An analogous question for the integer partition function reduces to a finite computation but in the setting where $ K $ varies, it becomes non-trivial. Our approach was to first study partitions into indecomposables because the other partitions can be built from them by combining the indecomposable parts together. Since the only indecomposable element in $ \Z_{\geq 1} $ is $ 1 $, studying integer partitions with indecomposable parts would not be very interesting. The new results in this paper, e.g., \Cref{thm:S1} and \Cref{thm:E6}, show how certain properties of $ p_K $ for $ K = \Q(\sqrt{D}) $ depend on the continued fraction of $ \omega_D $ (defined in \Cref{sec:Prelim}).

The rest of the paper is organized as follows. In \Cref{sec:Prelim}, we collect the preliminaries about continued fractions and indecomposables. In \Cref{sec:Ind}, we find all possible representations of elements in a certain specific form as a sum of indecomposables. These results are applied to prove \Cref{thm:S} in \Cref{sec:S} and \Cref{thm:N} in \Cref{sec:N}. Elements $ \alpha $ such that $ p_K(\alpha|\ind) = 2 $ are characterized in \Cref{sec:D2} and $ \calD(6) $ is determined in \Cref{sec:PK}.

\section*{Acknowledgments}
I would like to thank my PhD advisor V\'{i}t\v{e}zslav Kala for his guidance and support. I also thank Se Wook Jang for sending me his unpublished manuscript \cite{Jan} and the anonymous referee for useful comments, especially the suggestion to also provide an upper bound for the norm of an element with $ m $ partitions.

\section{Preliminaries}
\label{sec:Prelim}

We always work in a real quadratic field $ K = \Q(\sqrt{D}) $, where $ D \in \Z_{\geq 2} $ is squarefree. The Galois conjugate of $ \alpha \in K $ is denoted by $ \alpha' $, and we let $ \Nm(\alpha) = \alpha\alpha' $ and $ \Tr(\alpha) = \alpha+\alpha' $ be the \emph{norm} and \emph{trace} of $ \alpha $. The element $ \alpha $ is \emph{totally positive} if $ \alpha > 0 $ and $ \alpha' > 0 $, and  we denote this fact by $ \alpha \succ 0 $. We also write $ \alpha \succ \beta $ if $ \alpha-\beta \succ 0 $ and $ \alpha \succeq \beta $ if $ \alpha \succ \beta $ or $ \alpha = \beta $. The relation $ \succeq $ is a partial order on $ K $.

The pair $ (1, \omega_D) $, where
\[
	\omega_D := \begin{cases}
		\sqrt{D}&\text{if $ D \equiv 2, 3 \pmod{4} $,}\\
		\frac{1+\sqrt{D}}{2}&\text{if $ D \equiv 1 \pmod{4} $,}
	\end{cases}
\]
is a $ \Z $-basis of $ \O_K $. We also set
\[
	\xi_D := -\omega_D' = \begin{cases}
		\sqrt{D}&\text{if $ D \equiv 2, 3 \pmod{4} $,}\\
		\frac{\sqrt{D}-1}{2}&\text{if $ D \equiv 1 \pmod{4} $,}
	\end{cases}
\]
and $ \sigma_D := \lfloor\xi_D\rfloor+\omega_D $. Next, we describe how to characterize the indecomposable elements of $ K $ in terms of the continued fraction of $ \sigma_D $, following the exposition in \cite[p. 3]{HK}. The continued fraction expansion $ \sigma_D = [\overline{u_0,u_1,\dots,u_{s-1}}] $ is purely periodic. We have $ \lfloor\omega_D\rfloor = \lceil\frac{u_0}{2}\rceil = \frac{u_0+\Tr(\omega_D)}{2} $, hence $ \omega_D $ has a continued fraction expansion $ \omega_D = [\lceil u_0/2 \rceil, \overline{u_1, \dots, u_s}] $, where $ u_s = u_0 $. Let $ \frac{p_i}{q_i} $ be the \emph{convergents} to $ \omega_D $ defined by $ \frac{p_i}{q_i} := [\lceil u_0/2 \rceil, u_1, \dots, u_i] $ for $ i \geq 0 $. We have the recurrence relations
\begin{align*}
	p_{i+2}& = u_{i+2}p_{i+1}+p_i,\\
	q_{i+2}& = u_{i+2}q_{i+1}+q_i
\end{align*}
for $ i \geq -1 $ with the initial conditions $ (p_{-1},q_{-1}) = (1,0) $ and $ (p_0,q_0) = \left(\lceil u_0/2\rceil, 1\right) $. We let $ \alpha_i := p_i+q_i\xi_D $ for $ i \geq -1 $ and $ \alpha_{i,r} := \alpha_i+r\alpha_{i+1} $ for $ r \in \Z_{\geq 0} $. By an abuse of terminology, the $ \alpha_i $'s are also called \emph{convergents} and the $ \alpha_{i,r} $'s are called \emph{semiconvergents}. The sequence $ (\alpha_i) $ satisfies
\[
	\alpha_{i+2} = u_{i+2}\alpha_{i+1}+\alpha_i
\]
for $ i \geq -1 $ with the initial conditions $ \alpha_{-1} = 1 $ and $ \alpha_0 = \lceil u_0/2 \rceil+\xi_D = \lfloor \omega_D \rfloor+\xi_D $. It follows that $ \alpha_{i,u_{i+2}} = \alpha_{i+2,0} $. We have $ \alpha_i \succ 0 $ if and only if $ i $ is odd. Dress and Scharlau \cite{DS} proved that all the indecomposable elements of $ K $ are $ \alpha_{i,r} $, where $ i \geq -1 $ is odd and $ 0 \leq r \leq u_{i+2}-1 $ together with their conjugates $ \alpha_{i,r}' $.

Let $ \varepsilon > 1 $ denote the fundamental unit in $ \O_K $ and $ \varepsilon^+ > 1 $ the smallest totally positive unit in $ \O_K $. We have $ \varepsilon = \alpha_{s-1} $ and
\[
	\varepsilon^+ = \begin{cases}
		\varepsilon = \alpha_{s-1}&\text{if $ s $ is even,}\\
		\varepsilon^2 = \alpha_{2s-1}&\text{if $ s $ is odd.}
	\end{cases}
\]
It is not difficult to show that $ \varepsilon\alpha_i = \alpha_{s+i} $, hence $ \varepsilon^+\alpha_i = \alpha_{s+i} $ if $ s $ is even and $ \varepsilon^+\alpha_i = \alpha_{2s+i} $ if $ s $ is odd. Consequently, there are only finitely many indecomposables up to multiplication by totally positive units.

If we let
\[
	\calA := \{\alpha_{i,r}\,\vert\,i\geq -1\text{ odd}, 0 \leq r \leq u_{i+2}-1\}\setminus\{1\}
\]
and $ \calA' := \{\alpha'\,|\,\alpha\in\calA\} $, then the set of indecomposables equals $ \calI := \calA'\cup\{1\}\cup\calA $. The indecomposables $ \alpha_{i,r} $ get larger when $ (i,r) $ increases with respect to the lexicographical order, while $ \alpha_{i,r}' $ get smaller. We order the elements of $ \calI $ into a two-sided sequence
\[
	\dots < \beta_{-2} < \beta_{-1} < \beta_0 = 1 < \beta_1 < \beta_2 < \dots,
\] 
so that $ \beta_{-j} = \beta_j' $ for $ j \in \Z $.

\begin{lemma}[{\cite[Lemma 1]{HK}}]
\label{lem:vj}
For each $ j \in \Z $ we have that
\[
	v_j\beta_j = \beta_{j-1}+\beta_{j+1},
\]
where
\[
	v_j := \begin{cases}
		2&\text{if $ \beta_{|j|} = \alpha_{i,r} $ with odd $ i \geq -1 $ and $ 1 \leq r \leq u_{i+2}-1 $,}\\
		u_{i+1}+2&\text{if $ \beta_{|j|} = \alpha_{i,0} $ with odd $ i \geq -1 $.}
	\end{cases}
\]
\end{lemma}

Much of our paper is motivated by the following result of Hejda and Kala \cite{HK}. It was proved independently by Se Wook Jang in his unpublished manuscript \cite{Jan}.

\begin{theorem}[{\cite[Theorem 3]{HK}}]
\label{thm:HK3}
If $ \alpha \in \O_K^+ $, then there exist unique $ j_0, e, f \in \Z $ with $ e \geq 1 $ and $ f \geq 0 $ such that $ \alpha = e\beta_{j_0}+f\beta_{j_0+1} $.

Every relation of the form $ \sum h_j\beta_j= 0 $ (with $ h_j \in \Z $ and only finitely many non-zero) is a $ \Z $-linear combination of the relations $ \beta_{j-1}-v_j\beta_j+\beta_{j+1} = 0 $, where $ j \in \Z $.
\end{theorem}

The uniquely decomposable elements are characterized as follows.

\begin{theorem}[{\cite[Theorem 4]{HK}}]
\label{thm:D1}
Let $ \alpha \in \O_K^{+} $ and $ j \in \Z $, $ e \in \Z_{\geq 1} $, $ f \in \Z_{\geq 0} $ such that $ \alpha = e\beta_j+f\beta_{j+1} $. We have $ p_K(\alpha|\ind) = 1 $ if and only if
\[
	1 \leq e \leq v_j-1,\quad 0 \leq f \leq v_{j+1}-1\quad\text{and}\quad (e,f) \neq (v_j-1, v_{j+1}-1).
\]
\end{theorem}
\begin{proof}
This is an equivalent statement of \cite[Theorem 4]{HK}, see also \cite[(8)]{HK}.
\end{proof}

Finally, we collect some elementary properties of partitions and the partition function in $ K $. A partition $ \lambda_1+\lambda_2+\dots+\lambda_{\ell} $ will also be denoted by $ (\lambda_1,\lambda_2,\dots,\lambda_{\ell}) $. Let us stress that this will always be viewed as an unordered tuple. The partition function satisfies $ p_K(\alpha') = p_K(\alpha) $ and if $ \eta $ is a totally positive unit, then $ p_K(\eta\alpha) = p_K(\alpha) $. The same is true for $ p_K(\cdot|\ind) $. We also have $ p_K(\alpha) < p_K(\beta) $ and $ p_K(\alpha|\ind) \leq p_K(\beta|\ind) $ if $ \alpha \prec \beta $.

\section{Partitions with indecomposable parts}
\label{sec:Ind}

From \Cref{thm:HK3}, we know that every $ \alpha \in \O_K^+ $ can be expressed in the form $ \alpha = e\beta_{j}+f\beta_{j+1} $ for a (unique) $ j \in \Z $. In this section, we determine $ p_K(\alpha|\ind) $ for particular choices of $ (e, f) $. This is an intermediate step in proving some of our theorems. Throughout the paper, $ v_j \in \Z_{\geq 2} $ for $ j \in \Z $ are the numbers from \Cref{lem:vj}.

\begin{lemma}
\label{lem:EF1}
Let $ t \in \Z_{\geq 0} $. If there exists $ k_0 \in \Z_{\geq 0} $ such that $ v_k = 2 $ for $ k \in \{j-k_0, \dots, j+k_0+t\} $, then
\[
	\beta_j+\beta_{j+t} = \beta_{j-1}+\beta_{j+1+t} = \dots = \beta_{j-k_0-1}+\beta_{j+k_0+1+t}.
\]
\end{lemma}
\begin{proof}
By \Cref{lem:vj}, we have $ 2\beta_k = \beta_{k-1}+\beta_{k+1} $ for $ k \in \{j-k_0, \dots, j+k_0+t\} $. Let $ k_1 \in \{0, \dots, k_0\} $. If $ (k_1,t) \neq (0,0) $, then summing over $ k \in \{j-k_1, \dots, j+k_1+t\} $, we get
\[
	\sum_{k = j-k_1}^{j+k_1+t}2\beta_k = \sum_{k = j-k_1}^{j+k_1+t}\left(\beta_{k-1}+\beta_{k+1}\right) = \beta_{j-k_1-1}+\beta_{j-k_1}+\left(\sum_{k=j-k_1+1}^{j+k_1-1+t}2\beta_k\right)+\beta_{j+k_1+t}+\beta_{j+k_1+1+t},
\]
hence
\[
	\beta_{j-k_1}+\beta_{j+k_1+t} = \beta_{j-k_1-1}+\beta_{j+k_1+1+t}.
\]
But the last equality holds also for $ (k_1,t) = (0,0) $.
\end{proof}

\begin{lemma}
\label{lem:V2}
Let $ \alpha = \beta_{j_1}+\beta_{j_2} $, where $ j_1, j_2 \in \Z $ and $ j_1 \leq j_2 $. If $ \beta_{j_3} \preceq \alpha $ for some $ j_3 \in \Z $ such that $ j_3 > j_2 $ or $ j_3 < j_1 $, then $ v_k = 2 $ for $ k \in \{j_1, \dots, j_2\} $.
\end{lemma}
\begin{proof}
We may assume that $ \beta_{j_3} \preceq \alpha $ for some $ j_3 > j_2 $ by considering $ \alpha' $ and $ \beta_{j_3}' $. We have $ \beta_{j_2+1} \leq \beta_{j_3} \leq \alpha $. From \Cref{lem:vj}, we get
\[
	v_{j_2}\beta_{j_2} = \beta_{j_2-1}+\beta_{j_2+1} \leq \beta_{j_2-1}+\beta_{j_1}+\beta_{j_2},
\]
hence
\[
	(v_{j_2}-1)\beta_{j_2} \leq \beta_{j_1}+\beta_{j_2-1}.
\]
If $ v_{j_2} > 2 $, then
\[
	(v_{j_2}-1)\beta_{j_2} \geq 2\beta_{j_2} > \beta_{j_1}+\beta_{j_2-1},
\]
a contradiction. Thus, $ v_{j_2} = 2 $ and $ \beta_{j_2} \leq \beta_{j_1}+\beta_{j_2-1} $.

We claim that
\begin{equation}
\label{eq:V2}
	v_k = 2\quad\text{and}\quad \beta_k \leq \beta_{j_1}+\beta_{k-1},\qquad k \in \{j_1, \dots, j_2\}. 
\end{equation}
We showed this for $ k = j_2 $. Assume that \eqref{eq:V2} holds for $ k \in \{j_1+1, \dots, j_2\} $. From \Cref{lem:vj}, we get
\[
	v_{k-1}\beta_{k-1} = \beta_{k-2}+\beta_k \leq \beta_{k-2}+\beta_{j_1}+\beta_{k-1},
\]
hence
\[
	(v_{k-1}-1)\beta_{k-1} \leq \beta_{j_1}+\beta_{k-2}.
\]
If $ v_{k-1} > 2 $, then
\[
	(v_{k-1}-1)\beta_{k-1} \geq 2\beta_{k-1} > \beta_{j_1}+\beta_{k-2},
\]
where we used $ k-1 \geq j_1 $. This is a contradiction, hence $ v_{k-1} = 2 $ and $ \beta_{k-1} \leq \beta_{j_1}+\beta_{k-2} $, which proves \eqref{eq:V2} with $ k-1 $ in place of $ k $.
\end{proof}

\begin{lemma}
\label{lem:EF2}
Let $ t \in \{0,1\} $ and $ \alpha = \beta_j+\beta_{j+t} $. If $ v_j > 2 $ or $ v_{j+t} > 2 $, then $ p_K(\alpha|\ind) = 1 $. On the other hand, if there exists $ k_0 \in \Z_{\geq 0} $ such that $ v_k = 2 $ for $ k \in \{j-k_0, \dots, j+k_0+t\} $ but $ v_{j-k_0-1} > 2 $ or $ v_{j+k_0+1+t} > 2 $, then $ p_K(\alpha|\ind) = k_0+2 $. Moreover, all the partitions of $ \alpha $ with indecomposable parts are $ (\beta_{j-k},\beta_{j+k+t}) $, where $ k \in \{0,1,\dots, k_0+1\} $.
\end{lemma}
\begin{proof}
If $ t = 0 $ and $ v_j > 2 $, then $ \alpha = 2\beta_j $ is uniquely decomposable by \Cref{thm:D1}. Similarly, if $ t = 1 $ and $ v_j > 2 $ or $ v_{j+1} > 2 $, then $ \alpha = \beta_j+\beta_{j+1} $ is uniquely decomposable. In both cases, $ p_K(\alpha|\ind) = 1 $.

Assume that there exists $ k_0 \in \Z_{\geq 0} $ with the required properties. \Cref{lem:EF1} shows that $ (\beta_{j-k},\beta_{j+k+t}) $ for $ k \in \{0, 1, \dots, k_0+1\} $ are partitions of $ \alpha $, hence $ p_K(\alpha,\ind) \geq k_0+2 $. It remains to show that $ \alpha $ cannot be expressed as a sum of indecomposables in any other way.

Let $ j_1 := j-k_0-1 $ and $ j_2 = j+k_0+1+t $. Assume for contradiction that $ \beta_{j_3} \preceq \alpha $ for some $ j_3 \in \Z $ such that $ j_3 > j_2 $ or $ j_3 < j_1 $. Since $ \alpha = \beta_{j_1}+\beta_{j_2} $,
\Cref{lem:V2} implies $ v_k = 2 $ for $ k \in \{j_1, \dots, j_2\} $. In particular $ v_{j-k_0-1} = 2 $ and $ v_{j+k_0+1+t} = 2 $, contradicting the assumptions.

Finally, suppose that $ \beta_{j_4} $ with $ j_1 \leq j_4 \leq j_2 $ appears in some partition of $ \alpha $. If $ j_4 \leq j $, then we let $ j_4 = j-k $ and if $ j_4 > j $, then we let $ j_4 = j+k+t $ for $ k \in \{0,1,\dots,k_0+1\} $ (here we use the assumption $ t \in \{0,1\} $). Since $ \beta_{j-k} $ and $ \beta_{j+k+t} $ are indecomposable, the only partition of $ \alpha $ containing these elements is $ (\beta_{j-k}, \beta_{j+k+t}) $. This proves $ p_K(\alpha|\ind) = i_0+2 $.
\end{proof}

\begin{proposition}
\label{prop:EF}
Let $ i \geq -1 $ be odd, $ 0 \leq r \leq u_{i+2}-1 $, and $ \beta_j = \alpha_{i,r} $. We have
\begin{align*}
	p_K(2\beta_j|\ind)& = \min\{r+1, u_{i+2}-r+1\},\\
	p_K(\beta_j+\beta_{j+1}|\ind)& = \min\{r+1, u_{i+2}-r\}
\end{align*}
and
\begin{align*}
	p_K(2\beta_j)& = \min\{r+2,u_{i+2}-r+2\},\\
	p_K(\beta_j+\beta_{j+1})& = \min\{r+2,u_{i+2}-r+1\}.
\end{align*}
\end{proposition}
\begin{proof}
First, consider the case $ r = 0 $. We have $ \beta_j = \alpha_{i,0} $, hence $ v_j = u_{i+2}+2 > 2 $. By \Cref{lem:EF2}, $ p_K(2\beta_j|\ind) = 1 $ and $ p_K(\beta_j+\beta_{j+1}|\ind) = 1 $.

Secondly, consider the case $ 1 \leq r \leq u_{i+2}-2 $. We have $ \beta_j = \alpha_{i,r} $ and $ \beta_{j+1} = \alpha_{i,r+1} $. Moreover, $ \beta_{j-r} = \alpha_{i,0} $ and $ \beta_{j+(u_{i+2}-r)} = \alpha_{i,u_{i+2}} = \alpha_{i+2,0} $. From \Cref{lem:vj}, we obtain $ v_{j-r} = u_{i+1}+2 > 2 $, $ v_k = 2 $ for $ j-(r-1) \leq k \leq j+(u_{i+2}-r-1) $ and $ v_{j+(u_{i+2}-r)} = u_{i+3}+2 > 2 $. Thus, if $ t = 0 $, then the number $ k_0 $ in \Cref{lem:EF2} is $ k_0 = \min\{r-1, u_{i+2}-r-1\} $ and $ p_K(2\beta_j|\ind) = k_0+2 $. Similarly, if $ t = 1 $, then the number $ k_0 $ in \Cref{lem:EF2} is $ k_0 = \min\{r-1, u_{i+2}-r-2\} $ and $ p_K(\beta_j+\beta_{j+1}|\ind) = k_0+2 $.

Next, consider the case $ r = u_{i+2}-1 $. If $ u_{i+2} = 1 $, then $ r = 0 $, which was treated above. Hence, we can assume $ u_{i+2} \geq 2 $. We have $ \beta_j = \alpha_{i,r} $ and $ \beta_{j+1} = \alpha_{i,u_{i+2}} = \alpha_{i+2,0} $. From \Cref{lem:vj}, we obtain $ v_j = 2 $ and $ v_{j+1} = u_{i+3}+2 > 2 $. Thus, if $ t = 0 $, then the number $ k_0 $ in \Cref{lem:EF2} equals $ 0 $ and $ p_K(2\beta_j|\ind) = 2 $, while if $ t = 1 $, then \Cref{lem:EF2} implies $ p_K(\beta_j+\beta_{j+1}|\ind) = 1 $.

The only partitions of $ 2\beta_j $ with indecomposable parts are of the form $ (\beta_{j-k},\beta_{j+k}) $, where $ k \in \Z_{\geq 0} $. Hence, the only other partition of $ 2\beta_j $ is the trivial partition $ (2\beta_j) $, and $ p_K(2\beta_j) = p_K(2\beta_j|\ind)+1 $. Analogously, $ p_K(\beta_j+\beta_{j+1}) = p_K(\beta_j+\beta_{j+1}|\ind)+1 $.
\end{proof}

\section{Subsets of the range of the partition function}
\label{sec:S}

Our next theorem gives information about the range $ p_K\left(\O_K^+\right) $ of the partition function $ p_K $. We show that if there exists a large coefficient $ u_i $ for some odd $ i \geq 1 $, then the range contains a large set of consecutive integers. As a corollary, we obtain that $ \{1,2,\dots,m\} $ is a subset of the range for ``almost all'' real quadratic fields $ K $.

\begin{theorem}
\label{thm:S1}
Let $ \omega_D = [\lceil u_0/2 \rceil; \overline{u_1, \dots, u_{s}}] $ be the continued fraction expansion of $ \omega_D $. If $ B := \max\{u_i\,|\,i \geq 1\text{ odd}\} $, then
\[
	S_1 := \left\{1,2,\dots,\left\lfloor\frac{B}{2}\right\rfloor+2\right\} \subset p_K\left(\O_K^+\right)
\]
and
\[
	S_2 := \left\{1,2,\dots,\left\lfloor\frac{B}{2}\right\rfloor+1\right\} \subset p_K\left(\O_K^+|\ind\right).
\]
In other words, there exists an element $ \alpha \in \O_K^+ $ with $ m $ partitions for every $ m $ in $ S_1 $, and an element with $ m $ representations as a sum of indecomposables for every $ m $ in $ S_2 $.
\end{theorem}
\begin{proof}
Let $ i \geq -1 $ be odd and such that $ u_{i+2} = B $. We note that $ 1 \in p_K\left(\O_K^+\right) $ because $ p_K(\beta) = p_K(\beta|\ind) = 1 $ for every indecomposable $ \beta \in \O_K^+ $. If $ 2 \leq m \leq \left\lfloor\frac{B}{2}\right\rfloor+2 $, then we set $ r = m-2 $, so that $ 0 \leq r \leq \frac{B}{2} $. By \Cref{prop:EF}, the element $ \beta_j = \alpha_{i,r} $ satisfies
\[
	p_K(2\beta_j) = \min\{r+2,B-r+2\} = r+2 = m.
\]
This shows $ S_1 \subset p_K(\O_K^+) $.

Next, we prove the second inclusion in the theorem. If $ 1 \leq m \leq \left\lfloor\frac{B}{2}\right\rfloor+1 $, then we set $ r = m-1 $, so that $ 0 \leq r \leq \frac{B}{2} $. By \Cref{prop:EF}, the element $ \beta_j = \alpha_{i,r} $ satisfies
\[
	p_K(2\beta_j|\ind) = \min\{r+1,B-r+1\} = r+1 = m.
\]
This shows $ S_2 \subset p_K(\O_K^+) $.
\end{proof}

Kala, Yatsyna, and \.{Z}mija \cite{KYZ} showed that if $ B > 0 $ is a fixed bound, then the set of $ D $'s such that $ u_i \leq B $ for every odd $ i \geq 1 $ has density zero.

\begin{lemma}[{\cite[Corollary 2.12]{KYZ}}]
\label{lem:D0}
For every $ X, B \geq 2 $ satisfying $ X \geq B^{12}(\log X)^4 $, we have
\[
	\#\left\{1 \leq D \leq X\,\vert\,\omega_D = \left[\lceil u_0/2 \rceil, u_1,u_2,\dots\right], u_{2n-1} \leq B\text{ for all n}\right\} < 100B^{3/2}(\log X)^{3/2}X^{7/8}.
\]
\end{lemma}

As a straightforward application of \Cref{thm:S1} and \Cref{lem:D0}, we show the following result, which includes our main \Cref{thm:S}.

\begin{theorem}
\label{thm:S2}
Let $ m \in \Z $, $ m \geq 4 $, and
\[
	\calE_1(m, X) := \left\{2 \leq D \leq X\,|\,D\text{ squarefree}, \{1,2,\dots,m\}\not\subset p_K(\O_K^+)\text{ for }K = \Q(\sqrt{D})\right\},
\]
For every $ X \geq 2 $ satisfying $ X \geq (2m-5)^{12}(\log X)^4 $, we have
\[
	\#\calE_1(m, X) < 100(2m-5)^{3/2}(\log X)^{3/2}X^{7/8}.
\]

Secondly, let $ m \in \Z $, $ m \geq 3 $, and
\[
	\calE_2(m, X) := \left\{2 \leq D \leq X\,|\,D\text{ squarefree}, \{1,2,\dots,m\}\not\subset p_K(\O_K^+|\ind)\text{ for }K = \Q(\sqrt{D})\right\}.
\]
For every $ X \geq 2 $ satisfying $ X \geq (2m-3)^{12}(\log X)^4 $, we have
\[
	\#\calE_2(m,X) < 100(2m-3)^{3/2}(\log X)^{3/2}X^{7/8}.
\]
\end{theorem}
\begin{proof}
As above, we let $ \omega_D = [\lceil u_0/2 \rceil, \overline{u_1, \dots, u_s}] $ be the continued fraction expansion of $ \omega_D $.

First, let $ B_1 := 2m-4 $, so that $ \lfloor B_1/2 \rfloor+2 = m $. \Cref{thm:S1} shows that if there exists an odd $ i \geq 1 $ such that $ u_i \geq B_1 $, then $ \{1,2,\dots,m\} \subset p_K(\O_K^+) $ for $ K = \Q(\sqrt{D}) $. Thus,
\[
	\calE_1(m, X) \subset \left\{2 \leq D \leq X\,\vert\,\omega_D = \left[\lceil u_0/2 \rceil, u_1,u_2,\dots\right], u_{2n-1} \leq B_1-1\text{ for all }n\right\}.
\]
Since $ m \geq 4 $, we have $ B_1-1 = 2m-5 \geq 2 $. For $ X \geq 2 $ satisfying $ X \geq (B_1-1)^{12}(\log X)^4 $, the size of this set is $ < 100(B_1-1)^{3/2}(\log X)^{3/2}X^{7/8} $ by \Cref{lem:D0}.

Secondly, let $ B_2 := 2m-2 $, so that $ \lfloor B_2/2\rfloor+1 = m $. \Cref{thm:S1} shows that if there exists an odd $ i \geq 1 $ such that $ u_i \geq B_2 $, then $ \{1,2,\dots,m\} \subset p_K(\O_K^+|\ind) $ for $ K = \Q(\sqrt{D}) $. Thus,
\[
	\calE_2(m, X) \subset \left\{2 \leq D \leq X\,\vert\,\omega_D = \left[\lceil u_0/2 \rceil, u_1,u_2,\dots\right], u_{2n-1} \leq B_2-1\text{ for all }n\right\}.
\]
Since $ m \geq 3 $, we have $ B_2-1 = 2m-3 \geq 2 $. For $ X \geq 2 $ satisfying $ X \geq (B_2-1)^{12}(\log X)^4 $, the size of this set is $ < 100(B_2-1)^{3/2}(\log X)^{3/2}X^{7/8} $ by \Cref{lem:D0}.
\end{proof}

\Cref{thm:S2} shows in particular that for a fixed $ m \in \Z $, $ m \geq 4 $, we have
\[
	\lim_{X \to +\infty}\frac{\#\calE_1(m,X)}{\#\{2 \leq D \leq X\,|\,D\text{ squarefree}\}} = 0.
\]
We recall from the Introduction that
\[
	\calD(m) := \left\{D \in \Z_{\geq 2}\text{ squarefree}\ |\ m \notin p_K(\O_K^+)\right\}
\]
and that $ \calD(1) = \emptyset $, $ \calD(2) = \emptyset $, and $ \calD(3) = \{5\} $. Hence, $ \{1,2\} \subset p_K(\O_K^+) $ for every $ K $ and $ \{1,2,3\} \subset p_K(\O_K^+) $ for every $ K \neq \Q(\sqrt{5}) $. This leads to the following corollary.

\begin{corollary}
\label{cor:Dm0}
Let $ m \in \Z $, $ m \geq 1 $. The set $ \calD(m) $ has density zero, i.e.,
\[
	\lim_{X \to +\infty}\frac{\#\{2 \leq D \leq X,\,|\,D\text{ squarefree}, D \in \calD(m)\}}{\#\{2 \leq D \leq X\,|\,D\text{ squarefree}\}} = 0.
\]
\end{corollary}

\section{Norm bounds}
\label{sec:N}

The aim of this section is to prove \Cref{thm:N}. First, we introduce some useful notation and results from \cite{HK}.

We recall that $ \omega_D = [\lceil u_0/2 \rceil, \overline{u_1, \dots, u_s}] $ is the continued fraction expansion of $ \omega_D $. Let $ \gamma_0 = \omega_D $ and $ \gamma_i = [u_i,u_{i+1},u_{i+2},\dots] $ for $ i \geq 1 $. We have $ u_i < \gamma_i = u_i+\frac{1}{\gamma_{i+1}} < u_i+1 $ for $ i \geq 1 $. Moreover, $ u_0 = 2\lfloor\omega_D\rfloor-\Tr(\omega_D) < \sqrt{\Delta} $, hence $ u_i \leq u_0 < \sqrt{\Delta} $ for every $ i \geq 0 $.

If $ i \geq -1 $, then we let (as in \cite[p. 4]{HK})
\[
	N_i = |\Nm(\alpha_i)| = (-1)^{i+1}\Nm(\alpha_i) = \begin{cases}
		|p_i^2-Dq_i^2|&\text{if $ D \equiv 2, 3 \pmod{4} $,}\\
		|p_i^2-p_iq_i-q_i^2\frac{D-1}{4}|&\text{if $ D \equiv 1 \pmod{4} $.}
	\end{cases}
\]

\begin{lemma}[{\cite[Lemma 7]{HK}}]
\label{lem:Na1}
For all $ i \geq -1 $, we have
\[
	N_{i+1} = \frac{\sqrt{\Delta}}{\gamma_{i+2}}-\frac{N_i}{\gamma_{i+2}^2}.
\]
\end{lemma}

We note that $ N_{-1} = 1 $ because $ \alpha_{-1} = 1 $, and by the preceding lemma, we have $ N_{i+1} < \frac{\sqrt{\Delta}}{\gamma_{i+2}} < \frac{\sqrt{\Delta}}{u_{i+2}} $ for $ i \geq -1 $. Thus, $ N_i < \frac{\sqrt{\Delta}}{u_{i+1}} $ for every $ i \geq -1 $.

\begin{lemma}[{\cite[Lemma 8]{HK}}]
\label{lem:Na2}
If $ a, b \in \Z $ and $ i \geq -1 $ odd, then
\[
	\Nm\left(a\alpha_i+b\alpha_{i+1}\right) = \left(a-\frac{b}{\gamma_{i+2}}\right)\left(b\sqrt{\Delta}+aN_i-b\frac{N_i}{\gamma_{i+2}}\right).
\]
\end{lemma}

\begin{lemma}[{\cite[Proposition 9]{HK}}]
\label{lem:Na3}
If $ \alpha = e\alpha_{i,r}+f\alpha_{i,r+1} $, where $ i \geq -1 $ is odd, $ 0 \leq r \leq u_{i+2}-1 $, $ e \geq 1 $, and $ f \geq 0 $, then
\[
	\Nm(\alpha) < \sqrt{\Delta}\left((r+1)e+(r+2)f\right)(e+f)\qquad\text{and}\qquad \Nm(\alpha) < (e+f)^2\frac{\Delta}{4N_{i+1}}.
\]
\end{lemma}

The next lemma extends \Cref{lem:Na3} and its proof follows along the same lines.
\begin{lemma}
\label{lem:Na4}
If $ \alpha = e\alpha_{i,r}+f\alpha_{i,r+1} $, where $ i \geq -1 $ is odd, $ 0 \leq r \leq u_{i+2}-1 $, $ e \geq 1 $, and $ f \geq 0 $, then
\[
	\Nm(\alpha) < \sqrt{\Delta}\left((u_{i+2}-r+2)e+(u_{i+2}-r+1)f\right)\left((u_{i+2}-r+1)e+(u_{i+2}-r)f\right).
\]
\end{lemma}
\begin{proof}
We have $ \alpha = a\alpha_i+b\alpha_{i+1} $, where $ a = e+f $ and $ b = re+(r+1)f $. By \Cref{lem:Na2}, $ \Nm(\alpha) = \Nm(a\alpha_i+b\alpha_{i+1}) = AB $, where
\[
	A := a-\frac{b}{\gamma_{i+2}},\qquad B := b\sqrt{\Delta}+aN_i-b\frac{N_i}{\gamma_{i+2}}.
\]
Setting $ s := u_{i+2}-r $, we get
\[
	A = e\left(1-\frac{r}{\gamma_{i+2}}\right)+f\left(1-\frac{r+1}{\gamma_{i+2}}\right) = e\left(1-\frac{u_{i+2}-s}{\gamma_{i+2}}\right)+f\left(1-\frac{u_{i+2}-s+1}{\gamma_{i+2}}\right).
\]
From $ u_{i+2} < \gamma_{i+2} < u_{i+2}+1 $ it follows that
\[
	A < e\frac{s+1}{\gamma_{i+2}}+f\frac{s}{\gamma_{i+2}} < \frac{(s+1)e+sf}{u_{i+2}}.
\]
Since $ b = re+(r+1)f < (r+1)(e+f) $, we have
\[
	B < \sqrt{\Delta}(r+1)(e+f)+AN_i < \sqrt{\Delta}(u_{i+2}-s+1)(e+f)+\sqrt{\Delta}\frac{(s+1)e+sf}{u_{i+2}}.
\]
Thus,
\[
	\Nm(\alpha) = AB < \sqrt{\Delta}((s+1)e+sf)(e+f)+\sqrt{\Delta}\frac{((s+1)e+sf)^2}{u_{i+2}^2}.
\]
Since $ u_{i+2} \geq 1 $, we get
\[
	\Nm(\alpha) < \sqrt{\Delta}\left((s+2)e+(s+1)f\right)\left((s+1)e+sf\right).\qedhere
\]
\end{proof}

We are ready to prove \Cref{thm:N}, restated here for convenience.

\begin{theorem}
\label{thm:N'}
Let $ K = \Q(\sqrt{D}) $, where $ D \in \Z_{\geq 2} $ is squarefree and let $ m \in \Z_{\geq 1} $. If $ \alpha \in \O_K^+ $ can be represented as a sum of indecomposables in at most $ m $ ways, then
\[
	\Nm(\alpha) < m^2(2m+1)(2m+3)\cdot \sqrt{\Delta}\left(\sqrt{\Delta}+2\right)^2,
\]
where $ \Delta $ is the discriminant of $ K $.
\end{theorem}
\begin{proof}
By \Cref{thm:HK3}, there exist $ j, e, f \in \Z $ with $ e \geq 1 $ and $ f \geq 0 $ such that $ \alpha = e\beta_j+f\beta_{j+1} $. By passing to the conjugate if necessary, we can assume that $ \beta_j = \alpha_{i,r} $ for some $ i \geq -1 $ odd and $ 0 \leq r \leq u_{i+2}-1 $.

If $ e \geq mv_j $, then $ v_j\beta_j = \beta_{j-1}+\beta_{j+1} $ can be used to rewrite $ \alpha $ at least $ m $ times. More precisely,
\[
	\alpha = k\beta_{j-1}+(e-kv_j)\beta_j+(k+f)\beta_{j+1}
\]
for $ k \in \{0, 1, \dots, m\} $, hence $ p_K(\alpha|\ind) \geq m+1 $, a contradiction. Thus,\[
	e < mv_j \leq m(u_{i+1}+2) < m\left(\sqrt{\Delta}+2\right),
\]
where we used that $ v_j \leq u_{i+1}+2 $ by \Cref{lem:vj}.

Similarly, if $ f \geq mv_{j+1} $, then $ v_{j+1}\beta_{j+1} = \beta_j+\beta_{j+2} $ can be used to rewrite $ \alpha $ at least $ m $ times, a contradiction. Thus,
\[
	f < mv_{j+1} \leq m(u_{i+3}+2) < m\left(\sqrt{\Delta}+2\right).
\]

First, we consider the case $ f = 0 $. If $ e = 1 $, then $ \alpha $ is indecomposable and
\[
	\Nm(\alpha) = \Nm(\alpha_{i,r}) < \frac{\Delta}{4N_{i+1}}
\]
by \Cref{lem:Na3}. If $ e \geq 2 $, then by \Cref{prop:EF}, we have
\[
	m \geq p_K(\alpha|\ind) \geq p_K(2\beta_j|\ind) = \min\{r+1, u_{i+2}-r+1\}.
\]
Thus, $ r \leq m-1 $ or $ u_{i+2}-r \leq m-1 $. In the first case, we use \Cref{lem:Na3} to get
\[
	\Nm(\alpha) < \sqrt{\Delta}(r+1)e^2 < m^3\cdot\sqrt{\Delta}\left(\sqrt{\Delta}+2\right)^2.
\]
In the second case, we use \Cref{lem:Na4} to get
\[
	\Nm(\alpha) < \sqrt{\Delta}(u_{i+2}-r+2)(u_{i+2}-r+1)e^2 < \sqrt{\Delta}(m+1)me^2 < (m+1)m^3\cdot \sqrt{\Delta}\left(\sqrt{\Delta}+2\right)^2.
\]

Secondly, we consider the case $ f \geq 1 $. By \Cref{prop:EF},
\[
	m \geq p_K(\alpha|\ind) \geq p_K(\beta_j+\beta_{j+1}|\ind) = \min\{r+1,u_{i+2}-r\}.
\]
Thus, $ r \leq m-1 $ or $ u_{i+2}-r \leq m $. In the first case, we use \Cref{lem:Na3} to get
\[
	\Nm(\alpha) < \sqrt{\Delta}(me+(m+1)f)(e+f) < (2m^2+m)2m\cdot\sqrt{\Delta}\left(\sqrt{\Delta}+2\right)^2.
\]
In the second case, we use \Cref{lem:Na4} to get
\[
	\Nm(\alpha) < \sqrt{\Delta}\left((m+2)e+(m+1)f\right)\left((m+1)e+mf\right) < (2m^2+3m)(2m^2+m)\cdot\sqrt{\Delta}\left(\sqrt{\Delta}+2\right)^2.
\]
This proves the estimate in each case.
\end{proof}

Our main goal is to show that there exists a bound of the form $ \leq C(m)\Delta^{3/2} $ rather than find the best possible value for $ C(m) $, and \Cref{thm:N'} is in fact not optimal. For $ m = 1 $, we have
\[
	\Nm(\alpha) < \sqrt{\Delta}\left(2\sqrt{\Delta}+1\right)\left(3\sqrt{\Delta}+2\right)
\]
by \cite[Theorem 10]{HK} mentioned above. For $ m = 2 $, we will get an improvement in \Cref{thm:N2}.

Next, we use the same technique to prove a bound for the norm of $ \alpha \in \O_K^+ $ such that $ p_K(\alpha) = m $ when $ m \geq 2 $.

\begin{theorem}
\label{thm:Np}
Let $ K = \Q(\sqrt{D}) $, where $ D \in \Z_{\geq 2} $ is squarefree, let $ m \geq 2 $, and let $ n_0(m) $ be the largest $ n \in \Z_{\geq 2} $ such that $ p(n) \leq m $. If $ \alpha \in \O_K^+ $ satisfies $ p_K(\alpha) = m $, then
\[
    N(\alpha) \leq m(m+1) n_0(m)^2\cdot \sqrt{\Delta},
\]
where $ \Delta $ is the discriminant of $ K $.
\end{theorem}
\begin{proof}
As in the proof of \Cref{thm:N'}, we may assume that $ \alpha $ is of the form $ \alpha = e\beta_j+f\beta_{j+1} $, where $ j \geq 0 $, $ e \geq 1 $, $ f \geq 0 $, and $ \beta_j = \alpha_{i,r} $ for some $ i \geq -1 $ odd and $ 0 \leq r \leq u_{i+2}-1 $.

We claim that $ p(e+f) \leq p_K(\alpha) $. Let $ \varphi $ be a mapping which sends an integer partition $ \lambda = (\lambda_1,\dots,\lambda_{\ell}) $ of $ e+f $ to
\[
    \varphi(\lambda) := \left(\lambda_1\beta_j,\dots,\lambda_{s_1-1}\beta_j,\left(e-\sum_{s=1}^{s_1-1}\lambda_s\right)\beta_j+\left(\sum_{s=1}^{s_1}\lambda_s-e\right)\beta_{j+1}, \lambda_{s_1+1}\beta_{j+1},\dots,\lambda_{\ell}\beta_{j+1}\right),
\]
where $ 1 \leq s_1 \leq \ell $ is the largest index such that $ \sum_{s=1}^{s_1-1}\lambda_s \leq e $. The mapping $ \varphi $ is injective, which proves the claim. From the claim, we obtain $ e+f \leq n_0(m) $.

First, suppose that $ f = 0 $. Since $ m \geq 2 $, we have $ e \geq 2 $, and then by \Cref{prop:EF},
\[
    m = p_K(\alpha) \geq p_K(2\beta_j) = \min\{r+2,u_{i+2}-r+2\}.
\]
Thus, $ r \leq m-2 $ or $ u_{i+2}-r \leq m-2 $. In the first case, \Cref{lem:Na3} implies
\[
    \Nm(\alpha) < \sqrt{\Delta}(r+1)e^2 \leq (m-1)n_0(m)^2\cdot\sqrt{\Delta}.
\]
In the second case, \Cref{lem:Na4} implies
\[
    \Nm(\alpha) < \sqrt{\Delta}(u_{i+2}-r+2)(u_{i+2}-r+1)e^2 \leq m(m-1)n_0(m)^2\cdot\sqrt{\Delta}.
\]

Secondly, suppose that $ f \geq 1 $. By \Cref{prop:EF},
\[
    m = p_K(\alpha) \geq p_K(\beta_j+\beta_{j+1}) = \min\{r+2,u_{i+2}-r+1\}.
\]
Thus, $ r \leq m-2 $ or $ u_{i+2}-r \leq m-1 $. In the first case, \Cref{lem:Na3} implies
\[
    \Nm(\alpha) < \sqrt{\Delta}((m-1)e+mf)(e+f) < m n_0(m)^2\cdot\sqrt{\Delta}.
\]
In the second case, \Cref{lem:Na4} implies
\[
    \Nm(\alpha) < \sqrt{\Delta}((m+1)e+mf)(me+(m-1)f) < (m+1)m n_0(m)^2\cdot\sqrt{\Delta}.\qedhere
\]
\end{proof}

\section{Elements represented as a sum of indecomposables in two different ways}
\label{sec:D2}

Next, we prove a characterization of the elements $ \alpha \in \O_K^+ $ which can be expressed as a sum of indecomposables in exactly two ways.

\begin{theorem}
\label{thm:D2}
Let $ \alpha \in \O_K^{+} $ and $ j \in \Z $, $ e \in \Z_{\geq 1} $, $ f \in \Z_{\geq 0} $ such that $ \alpha = e\beta_j+f\beta_{j+1} $. We have $ p_K(\alpha|\ind) = 2 $ if and only if one of the following conditions is satisfied:
\begin{enumerate}
\item $ v_j \leq e \leq 2v_j-1 $, $ 0 \leq f \leq v_{j+1}-2 $, and
\[
	(e, f) \neq (2v_j-1, v_{j+1}-2),\quad (v_{j-1},e) \neq (2,2v_j-1), \quad (v_{j-1},e,f) \neq (2,2v_j-2,v_{j+1}-2),
\]
\item $ 1 \leq e \leq v_j-2 $, $ v_{j+1} \leq f \leq 2v_{j+1}-1 $, and
\[
	(e, f) \neq (v_j-2, 2v_{j+1}-1),\quad (f,v_{j+2}) \neq (2v_{j+1}-1,2),\quad (e,f,v_{j+2}) \neq (v_j-2, 2v_{j+1}-2,2),
\]
\item $ e = v_j-1 $, $ f = v_{j+1}-1 $, and $ (v_{j-1},v_j,v_{j+1},v_{j+2}) \neq (2,2,2,2) $.
\end{enumerate}
\end{theorem}
\begin{proof}
The theorem follows from \Cref{lem:D2nec} and \Cref{lem:D2suf1,lem:D2suf2,lem:D2suf3} below.
\end{proof}

\begin{lemma}
\label{lem:D2nec}
Let $ \alpha \in \O_K^{+} $ and $ j \in \Z $, $ e \in \Z_{\geq 1} $, $ f \in \Z_{\geq 0} $ such that $ \alpha = e\beta_j+f\beta_{j+1} $. If $ p_K(\alpha|\ind) = 2 $, then one of the conditions (1), (2), (3) in \Cref{thm:D2} is satisfied.
\end{lemma}
\begin{proof}
First, we show that if $ \alpha $ can be expressed as a sum of indecomposables in exactly $2$ ways, then either $ v_j \leq e \leq 2v_j-1 $ and $ 0 \leq f \leq v_{j+1}-2 $, or $ 1 \leq e \leq v_j-2 $ and $ v_{j+1} \leq f \leq 2v_{j+1}-1 $, or $ e = v_j-1 $ and $ f = v_{j+1}-1 $. If $ 1 \leq e \leq v_j-1 $, $ 0 \leq f \leq v_{j+1}-1 $, and $ (e,f) \neq (v_j-1, v_{j+1}-1) $, then $ p_K(\alpha|\ind) = 1 $ by \Cref{thm:D1}. On the other hand, if $ e \geq 2v_j $, then
\[
	\alpha = e\beta_j+f\beta_{j+1} = \beta_{j-1}+(e-v_j)\beta_j+(f+1)\beta_{j+1} = 2\beta_{j-1}+(e-2v_j)\beta_j+(f+2)\beta_{j+1},
\]
and if $ f \geq 2v_{j+1} $, then
\[
	\alpha = e\beta_j+f\beta_{j+1} = (e+1)\beta_j+(f-v_{j+1})\beta_{j+1}+\beta_{j+2} = (e+2)\beta_j+(f-2v_{j+1})\beta_{j+1}+2\beta_{j+2},
\]
hence $ p_K(\alpha|\ind) \geq 3 $ in both of these cases. If $ v_j \leq e \leq 2v_j-1 $ and $ v_{j+1}-1 \leq f $, then
\[
	\alpha = e\beta_j+f\beta_{j+1} = \beta_{j-1}+(e-v_j)\beta_j+(f+1)\beta_{j+1} = \beta_{j-1}+(e-v_j+1)\beta_j+(f+1-v_{j+1})\beta_{j+1}+\beta_{j+2},
\]
hence $ p_K(\alpha|\ind) \geq 3 $. Similarly, if $ v_j-1 \leq e $ and $ v_{j+1} \leq f \leq 2v_{j+1}-1 $, then
\[
	\alpha = e\beta_j+f\beta_{j+1} = (e+1)\beta_j+(f-v_{j+1})\beta_{j+1}+\beta_{j+2} = \beta_{j-1}+(e+1-v_j)\beta_j+(f-v_{j+1}+1)\beta_{j+1}+\beta_{j+2},
\]
hence $ p_K(\alpha,\ind) \geq 3 $.

\

Secondly, we show that if $ (e, f) = (2v_j-1, v_{j+1}-2) $ or $ (v_{j-1}, e) = (2, 2v_j-1) $ or $ (v_{j-1}, e, f) = (2, 2v_j-2, 2v_{j+1}-2) $, then $ p_K(\alpha|\ind) \geq 3 $. If $ (e, f) = (2v_j-1, v_{j+1}-2) $, then
\[
	\alpha = (2v_j-1)\beta_j+(v_{j+1}-2)\beta_{j+1} = \beta_{j-1}+(v_j-1)\beta_j+(v_{j+1}-1)\beta_{j+1} = 2\beta_{j-1}+\beta_{j+2}.
\]
If $ (v_{j-1}, e) = (2, 2v_j-1) $, then
\[
	\alpha = (2v_j-1)\beta_j+f\beta_{j+1} = \beta_{j-1}+(v_j-1)\beta_j+(f+1)\beta_{j+1} = \beta_{j-2}+(f+2)\beta_{j+1}.
\]
If $ (v_{j-1}, e, f) = (2, 2v_j-2, v_{j+1}-2) $, then
\[
	\alpha = (2v_j-2)\beta_j+(v_{j+1}-2)\beta_{j+1} = \beta_{j-1}+(v_j-2)\beta_j+(v_{j+1}-1)\beta_{j+1} = \beta_{j-2}+\beta_{j+2}.
\]
Thus, $ p_K(\alpha|\ind) \geq 3 $ in all of these cases. Analogously, one can show that if $ (e, f) = (v_j-2, 2v_{j+1}-1) $ or $ (f,v_{j+2}) = (2v_{j+1}-1,2) $ or $ (e, f, v_{j+2}) = (v_j-2, 2v_{j+1}-2, 2) $, then $ p_K(\alpha|\ind) \geq 3 $.

\

Finally, we show that if $ e = v_j-1 $, $ f = v_{j+1}-1 $, and $ (v_{j-1}, v_j, v_{j+1}, v_{j+2}) = (2,2,2,2) $, then $ p_K(\alpha|\ind) \geq 3 $. We have $ \alpha = \beta_j+\beta_{j+1} $, and by \Cref{lem:EF1} with $ t = 1 $ and $ k_0 = 1 $,
\[
	\beta_j+\beta_{j+1} = \beta_{j-1}+\beta_{j+2} = \beta_{j-2}+\beta_{j+3},
\]
hence $ p_K(\alpha|\ind) \geq 3 $.
\end{proof}

\begin{lemma}
\label{lem:D2suf1}
Let $ \alpha \in \O_K^{+} $ and $ j \in \Z $, $ e \in \Z_{\geq 1} $, $ f \in \Z_{\geq 0} $ such that $ \alpha = e\beta_j+f\beta_{j+1} $. If (1) in \Cref{thm:D2} holds, i.e., $ v_j \leq e \leq 2v_j-1 $, $ 0 \leq f \leq v_{j+1}-2 $, and
\[
	(e, f) \neq (2v_j-1, v_{j+1}-2),\quad (v_{j-1},e) \neq (2,2v_j-1),\quad (v_{j-1},e,f) \neq (2,2v_j-2,v_{j+1}-2),
\]
then $ p_K(\alpha|\ind) = 2 $.
\end{lemma}
\begin{proof}
Assume that the condition holds. We have
\[
	\alpha = e\beta_j+f\beta_{j+1} = \beta_{j-1}+(e-v_j)\beta_j+(f+1)\beta_{j+1},
\]
hence $ p_K(\alpha|\ind) \geq 2 $. It remains to show that there are no other partitions of $ \alpha $ with indecomposable parts.

First, suppose for contradiction that there exists $ k \in \Z $, $ k \geq j+2 $ such that $ \beta_k \preceq \alpha $. We have $ \beta_{j+2} \leq \beta_k \leq \alpha $.

Case $ e \leq 2v_j-2 $ and $ f \leq v_{j+1}-3 $: we have
\[
	2\beta_j+3\beta_{j+1}+\beta_{j+2} \leq (e+2)\beta_j+(f+3)\beta_{j+1} \leq 2v_j\beta_j+v_{j+1}\beta_{j+1} = 2\beta_{j-1}+2\beta_{j+1}+\beta_j+\beta_{j+2},
\]
hence
\[
	\beta_j+\beta_{j+1} \leq 2\beta_{j-1},
\]
a contradiction.

Case $ e = 2v_j-1 $, $ f \leq v_{j+1}-3 $, and $ v_{j-1} \neq 2 $: we have
\[
	\beta_j+3\beta_{j+1}+\beta_{j+2} \leq (e+1)\beta_j+(f+3)\beta_{j+1} \leq 2v_j\beta_j+v_{j+1}\beta_{j+1} = 2\beta_{j-1}+2\beta_{j+1}+\beta_j+\beta_{j+2},
\]
hence $ \beta_{j+1} \leq 2\beta_{j-1} $. Since $ v_{j-1} \neq 2 $, we have $ v_{j-1} \geq 3 $, and so
\[
	2\beta_j \leq v_j\beta_j = \beta_{j-1}+\beta_{j+1} \leq 3 \beta_{j-1} \leq v_{j-1}\beta_{j-1} = \beta_{j-2}+\beta_j,
\]
hence $ \beta_j \leq \beta_{j-2} $, a contradiction.

Case $ e \leq 2v_j-3 $ and $ f = v_{j+1}-2 $: we have
\[
	3\beta_j+2\beta_{j+1}+\beta_{j+2} \leq (e+3)\beta_j+(f+2)\beta_{j+1} \leq 2v_j\beta_j+v_{j+1}\beta_{j+1} = 2\beta_{j-1}+2\beta_{j+1}+\beta_j+\beta_{j+2},
\]
hence $ 2\beta_j \leq 2\beta_{j-1} $, a contradiction.

Case $ e = 2v_j-2 $, $ f = v_{j+1}-2 $, and $ v_{j-1} \neq 2 $: we have
\[
	2\beta_j+2\beta_{j+1}+\beta_{j+2} \leq (e+2)\beta_j+(f+2)\beta_{j+1} = 2v_j\beta_j+v_{j+1}\beta_{j+1} = 2\beta_{j-1}+2\beta_{j+1}+\beta_j+\beta_{j+2},
\]
hence $ \beta_j \leq 2\beta_{j-1} $. Since $ v_{j-1} \neq 2 $, we have $ v_{j-1} \geq 3 $, and so
\[
	\beta_{j-1}+\beta_j \leq 3\beta_{j-1} \leq v_{j-1}\beta_{j-1} = \beta_{j-2}+\beta_j,
\]
hence $ \beta_{j-1} \leq \beta_{j-2} $, a contradiction.

\

Secondly, suppose for contradiction that there exists $ k \in \Z $, $ k \leq j-2 $, such that $ \beta_k \preceq \alpha $. We have $ \beta_{j-2}' \leq \beta_k' \leq \alpha' $.

Case $ e \leq 2v_j-2 $ and $ f \leq v_{j+1}-3 $: we have
\[
	\beta_{j-2}'+2\beta_j'+3\beta_{j+1}' \leq (e+2)\beta_j'+(f+3)\beta_{j+1}' \leq 2v_j\beta_j'+v_{j+1}\beta_{j+1}' = 2\beta_{j-1}'+2\beta_{j+1}'+\beta_j'+\beta_{j+2}',
\]
hence
\[
	\beta_{j-2}'+\beta_j'+\beta_{j+1}' \leq 2\beta_{j-1}'+\beta_{j+2}'.
\]
It follows that
\[
	2\beta_{j-1}'+\beta_{j+1}' \leq v_{j-1}\beta_{j-1}'+\beta_{j+1}' = \beta_{j-2}'+\beta_j'+\beta_{j+1}' \leq 2\beta_{j-1}'+\beta_{j+2}',
\]
hence $ \beta_{j+1}' \leq \beta_{j+2}' $, a contradiction.

Case $ e = 2v_j-1 $, $ f \leq v_{j+1}-3 $, and $ v_{j-1} \neq 2 $: we have
\[
	\beta_{j-2}'+\beta_j'+3\beta_{j+1}' \leq (e+1)\beta_j'+(f+3)\beta_{j+1}' \leq 2v_j\beta_j'+v_{j+1}\beta_{j+1}' = 2\beta_{j-1}'+2\beta_{j+1}'+\beta_j'+\beta_{j+2}',
\]
hence
\[
	\beta_{j-2}'+\beta_{j+1}' \leq 2\beta_{j-1}'+\beta_{j+2}'.
\]
Since $ v_{j-1} \geq 3 $, it follows that
\[
	3\beta_{j-1}'+\beta_{j+1}' \leq v_{j-1}\beta_{j-1}'+\beta_{j+1}' = \beta_{j-2}'+\beta_j'+\beta_{j+1}' \leq 2\beta_{j-1}'+\beta_j'+\beta_{j+2}',
\]
and so
\[
	2\beta_j' \leq v_j\beta_j' = \beta_{j-1}'+\beta_{j+1}' \leq \beta_j'+\beta_{j+2}',
\]
hence $ \beta_j' \leq \beta_{j+2}' $, a contradiction.

Case $ e \leq 2v_j-3 $ and $ f = v_{j+1}-2 $: we have
\[
	\beta_{j-2}'+3\beta_j'+2\beta_{j+1}' \leq (e+3)\beta_j'+(f+2)\beta_{j+1}' \leq 2v_j\beta_j'+v_{j+1}\beta_{j+1}' = 2\beta_{j-1}'+2\beta_{j+1}'+\beta_j'+\beta_{j+2}',
\]
hence
\[
	\beta_{j-2}'+2\beta_j' \leq 2\beta_{j-1}'+\beta_{j+2}'.
\]
It follows that
\[
	2\beta_{j-1}'+\beta_j' \leq v_{j-1}\beta_{j-1}'+\beta_j' = \beta_{j-2}'+2\beta_j' \leq 2\beta_{j-1}'+\beta_{j+2}',
\]
hence $ \beta_j' \leq \beta_{j+2}' $, a contradiction.

Case $ e = 2v_j-2 $, $ f = v_{j+1}-2 $, and $ v_{j-1} \neq 2 $: we have
\[
	\beta_{j-2}'+2\beta_j'+2\beta_{j+1}' \leq (e+2)\beta_j'+(f+2)\beta_{j+1}' = 2v_j\beta_j'+v_{j+1}\beta_{j+1}' = 2\beta_{j-1}'+2\beta_{j+1}'+\beta_j'+\beta_{j+2}',
\]
hence
\[
	\beta_{j-2}'+\beta_j' \leq 2\beta_{j-1}'+\beta_{j+2}'.
\]
Since $ v_{j-1} \geq 3 $, it follows that
\[
	3\beta_{j-1}' \leq v_{j-1}\beta_{j-1}' = \beta_{j-2}'+\beta_j' \leq 2\beta_{j-1}'+\beta_{j+2}',
\]
hence $ \beta_{j-1}' \leq \beta_{j+2}' $, a contradiction.

\

We showed that every partition of $ \alpha $ with indecomposable parts is of the form
\[
	\alpha = a_{j-1}\beta_{j-1}+a_j\beta_j+a_{j+1}\beta_{j+1},
\]
where $ a_{j-1}, a_j, a_{j+1} \in \Z_{\geq 0} $. Using $ \beta_{j-1} = v_j\beta_j-\beta_{j+1} $, we get
\[
	e\beta_j+f\beta_{j+1} = \left(a_{j-1}v_j+a_j\right)\beta_j+\left(a_{j+1}-a_{j-1}\right)\beta_{j+1}.
\]
The elements $ \beta_j $ and $ \beta_{j+1} $ are linearly independent over $ \Q $, hence $ e = a_{j-1}v_j+a_j $ and $ f = a_{j+1}-a_{j-1} $. Since $ v_j \leq e \leq 2v_j-1 $, the only possibilities for $ a_{j-1} $ are $ 0 $ or $ 1 $, which gives us $ (a_{j-1},a_j,a_{j+1}) = (0,e,f) $ or $ (a_{j-1},a_j,a_{j+1}) = (1,e-v_j,f+1) $. Thus, $ p_K(\alpha|\ind) = 2 $.
\end{proof}

\begin{lemma}
\label{lem:D2suf2}
Let $ \alpha \in \O_K^{+} $ and $ j \in \Z $, $ e \in \Z_{\geq 1} $, $ f \in \Z_{\geq 0} $ such that $ \alpha = e\beta_j+f\beta_{j+1} $. If (2) in \Cref{thm:D2} holds, i.e., $ 1 \leq e \leq v_j-2 $, $ v_{j+1} \leq f \leq 2v_{j+1}-1 $, and
\[
	(e, f) \neq (v_j-2, 2v_{j+1}-1),\quad (f,v_{j+2}) \neq (2v_{j+1}-1,2),\quad (e,f,v_{j+2}) \neq (v_j-2, 2v_{j+1}-2,2),
\]
then $ p_K(\alpha|\ind) = 2 $.
\end{lemma}
\begin{proof}
If we let $ j' = -(j+1) $, $ e' = f $, and $ f' = e $, then
\[
	\alpha' = e\beta_j'+f\beta_{j+1}' = f\beta_{-(j+1)}+e\beta_{-j} = e'\beta_{j'}+f'\beta_{j'+1}.
\]
Since $ v_{j'-1} = v_{j+2} $, $ v_{j'} = v_{j+1} $, and $ v_{j'+1} = v_j $, it follows that $ \alpha' $ satisfies condition (1) in \Cref{thm:D2} (with $ j' $, $ e' $, and $ f' $ in place of $ j $, $ e $, and $ f $). By \Cref{lem:D2suf1}, $ p_K(\alpha|\ind) = p_K(\alpha'|\ind) = 2 $.
\end{proof}

\begin{lemma}
\label{lem:D2suf3}
Let $ \alpha \in \O_K^{+} $ and $ j \in \Z $, $ e \in \Z_{\geq 1} $, $ f \in \Z_{\geq 0} $ such that $ \alpha = e\beta_j+f\beta_{j+1} $. If (3) in \Cref{thm:D2} holds, i.e., $ e = v_j-1 $, $ f = v_{j+1}-1 $, and $ (v_{j-1},v_j,v_{j+1},v_{j+2}) \neq (2,2,2,2) $, then $ p_K(\alpha|\ind) = 2 $.
\end{lemma}
\begin{proof}
Assume that the condition holds. We have
\[
	\alpha = (v_j-1)\beta_j+(v_{j+1}-1)\beta_{j+1} = \beta_{j-1}+\beta_{j+2},
\]
hence $ p_K(\alpha|\ind) \geq 2 $. It remains to show that $ p_K(\alpha|\ind) \leq 2 $.

Suppose for contradiction that $ \beta_{j_3} \preceq \alpha $ for some $ j_3 \in \Z $ such that $ j_3 > j+2 $ or $ j_3 < j-1 $. By \Cref{lem:V2}, $ v_k = 2 $ for $ k \in \{j-1,j,j+1,j+2\} $, a contradiction.

Since the elements $ \beta_{j-1} $ and $ \beta_{j+2} $ are indecomposable, the only partition of $ \alpha $ containing $ \beta_{j-1} $ or $ \beta_{j+2} $ is $ \beta_{j-1}+\beta_{j+2} $. This concludes the proof that $ p_K(\alpha|\ind) = 2 $.
\end{proof}

In \Cref{thm:D2b} below, we provide an explicit characterization of the elements $ \alpha \in \O_K^+ $ such that $ p_K(\alpha|\ind) = 2 $ expressed in terms of the $ u_i $'s instead of $ v_k $'s. Then we use it to improve the bound on $ \Nm(\alpha) $ from \Cref{thm:N} with $ m = 2 $.

Suppose that $ \alpha = e\alpha_{i,r}+f\alpha_{i,r+1} $, where $ i \geq -1 $ is odd and $ 0 \leq r \leq u_{i+2}-1 $. Let $ j \in \Z_{\geq 0} $ be such that $ \beta_j = \alpha_{i,r} $, hence $ \alpha = e\beta_j+f\beta_{j+1} $. By \Cref{lem:vj}, we have
\[
	v_j = \begin{cases}
		u_{i+1}+2&\text{if $ r = 0 $,}\\
		2&\text{if $ 1 \leq r \leq u_{i+2}-1 $,}
	\end{cases}\qquad
	v_{j+1} = \begin{cases}
		2&\text{if $ 0 \leq r \leq u_{i+2}-2 $,}\\
		u_{i+3}+2&\text{if $ r = u_{i+2}-1 $.}
	\end{cases}
\]
If $ r \geq 1 $, then $ \beta_{j-1} = \alpha_{i,r-1} $, and if $ r = 0 $ and $ i \geq 1 $, then $ \beta_{j-1} = \alpha_{i-2,u_i-1} $. On the other hand, if $ r = 0 $ and $ i = -1 $, then $ j = 0 $ and $ v_{j-1} = v_{-1} = v_1 $. Thus,
\[
	v_{j-1} = \begin{cases}
		2&\text{if $ r = 0 $ and $ u_{|i|} \geq 2 $,}\\
		u_{|i-1|}+2&\text{if $ r = 0 $ and $ u_{|i|} = 1 $,}\\
		u_{i+1}+2&\text{if $ r = 1 $,}\\
		2&\text{if $ 2 \leq r \leq u_{i+2}-1 $.}
	\end{cases}
\]
If $ r \leq u_{i+2}-3 $, then $ \beta_{j+2} = \alpha_{i,r+2} $, and if $ r = u_{i+2}-2 $, then $ \beta_{j+2} = \alpha_{i+2,0} $. If $ r = u_{i+2}-1 $ and $ u_{i+4} \geq 2 $, then $ \beta_{j+2} = \alpha_{i+2,1} $, and if $ r = u_{i+2}-1 $ and $ u_{i+4} = 1 $, then $ \beta_{j+2} = \alpha_{i+4,0} $. Thus,
\[
	v_{j+2} = \begin{cases}
		2&\text{if $ 0 \leq r \leq u_{i+2}-3 $,}\\
		u_{i+3}+2&\text{if $ r = u_{i+2}-2 $,}\\
		2&\text{if $ r = u_{i+2}-1 $ and $ u_{i+4} \geq 2 $,}\\
		u_{i+5}+2&\text{if $ r = u_{i+2}-1 $ and $ u_{i+4} = 1 $.}
	\end{cases}
\]

\begin{lemma}
\label{lem:D2b1}
Let $ \alpha = e\alpha_{i,r}+f\alpha_{i,r+1} $, where $ i \geq -1 $ is odd and $ 0 \leq r \leq u_{i+2}-1 $. The element $ \alpha $ satisfies condition (1) in \Cref{thm:D2} if and only if one of the following holds:
\begin{enumerate}[(a)]
\item[(a)] $ r = 0 $, $ u_{i+1}+2 \leq e \leq 2u_{i+1}+1 $, $ f = 0 $, and $ u_{i+2} \geq 2 $,
\item[(b)] $ r = 0 $, $ e = 2u_{i+1}+2 $, $ f = 0 $, $ u_{|i|} = 1 $, and $ u_{i+2} \geq 2 $,
\item[(c)] $ r = 0 $, $ u_{i+1}+2 \leq e \leq 2u_{i+1}+1 $, $ 0 \leq f \leq u_{i+3} $, $ u_{|i|} \geq 2 $, and $ u_{i+2} = 1 $,
\item[(d)] $ r = 0 $, $ e = 2u_{i+1}+2 $, $ 0 \leq f \leq u_{i+3}-1 $, $ u_{|i|} \geq 2 $, and $ u_{i+2} = 1 $,
\item[(e)] $ r = 0 $, $ u_{i+1}+2 \leq e \leq 2u_{i+1}+2 $, $ 0 \leq f \leq u_{i+3} $, $ u_{|i|} = 1 $, and $ u_{i+2} = 1 $,
\item[(f)] $ r = 0 $, $ e = 2u_{i+1}+3 $, $ 0 \leq f \leq u_{i+3}-1 $, $ u_{|i|} = 1 $, and $ u_{i+2} = 1 $,
\item[(g)] $ r = 1 $, $ e = 2 $, $ f = 0 $, and $ u_{i+2} \geq 3 $,
\item[(h)] $ r = 1 $, $ e = 2 $, $ 0 \leq f \leq u_{i+3} $, and $ u_{i+2} = 2 $,
\item[(i)] $ r = 1 $, $ e = 3 $, $ 0 \leq f \leq u_{i+3}-1 $, and $ u_{i+2} = 2 $,
\item[(j)] $ r = u_{i+2}-1 $, $ e = 2 $, $ 0 \leq f \leq u_{i+3}-1 $, and $ u_{i+2} \geq 3 $.
\end{enumerate}
\end{lemma}
\begin{proof}
Let $ j \in \Z_{\geq 0} $ be such that $ \beta_j = \alpha_{i,r} $, hence $ \alpha = e\beta_j+f\beta_{j+1} $. Assume that $ \alpha $ satisfies condition (1) in \Cref{thm:D2}, i.e., $ v_j \leq e \leq 2v_j-1 $, $ 0 \leq f \leq v_{j+1}-2 $, and
\[
	(e, f) \neq (2v_j-1,v_{j+1}-2),\quad (v_{j-1},e) \neq (2,2v_j-1),\quad (v_{j-1},e,f) \neq (2,2v_j-2,v_{j+1}-2).
\]

Case $ r = 0 $: we have $ v_j = u_{i+1}+2 $. If $ u_{i+2} \geq 2 $, then $ v_{j+1} = 2 $. If $ u_{|i|} \geq 2 $, then $ v_{j-1} = 2 $. Condition (1) becomes $ u_{i+1}+2 \leq e \leq 2u_{i+1}+3 $, $ f = 0 $, and $ e \neq 2u_{i+1}+3 $, $ e \neq 2u_{i+1}+2 $, hence $ u_{i+1}+2 \leq e \leq 2u_{i+1}+1 $. If $ u_{|i|} = 1 $, then $ v_{j-1} = u_{|i-1|}+2 > 2 $. Condition (1) becomes $ u_{i+1}+2 \leq e \leq 2u_{i+1}+3 $, $ f = 0 $, and $ e \neq 2u_{i+1}+3 $, hence we get the additional possibility $ e = 2u_{i+1}+2 $. This gives us (a) and (b).

If $ u_{i+2} = 1 $, then $ v_{j+1} = u_{i+3}+2 $. We get $ u_{i+1}+2 \leq e \leq 2u_{i+1}+3 $, $ 0 \leq f \leq u_{i+3} $. If $ u_{|i|} \geq 2 $, then $ v_{j-1} = 2 $ and $ e \neq 2u_{i+1}+3 $, $ (e, f) \neq (2u_{i+1}+2, u_{i+3}) $. This gives us (c) and (d). If $ u_{|i|} = 1 $, then $ v_{j-1} = u_{|i-1|}+2 > 2 $, hence $ (e,f) \neq (2u_{i+1}+3,u_{i+3}) $. This gives us (e) and (f).

Case $ r = 1 $: we must have $ u_{i+2} \geq 2 $ and $ v_j = 2 $. Moreover, $ v_{j-1} = u_{i+1}+2 > 2 $. If $ u_{i+2} \geq 3 $, then $ v_{j+1} = 2 $, hence $ 2 \leq e \leq 3 $, $ f = 0 $, and $ (e,f) \neq (3,0) $. This gives us (g).

If $ u_{i+2} = 2 $, then $ v_{j+1} = u_{i+3}+2 $, hence $ 2 \leq e \leq 3 $, $ 0 \leq f \leq u_{i+3} $, and $ (e,f) \neq (3,u_{i+3}) $. This gives us (h) and (i).

Case $ 2 \leq r \leq u_{i+2}-2 $: we have $ v_{j-1} = 2 $, $ v_j = 2 $, and $ v_{j+1} = 2 $, hence $ 2 \leq e \leq 3 $, $ f = 0 $, and $ e \neq 3 $, $ (e,f) \neq (2,0) $. We see that these conditions are never satisfied.

Case $ r = u_{i+2}-1 $: because we have already dealt with the cases $ r = 0 $ and $ r = 1 $, we may assume $ u_{i+2} \geq 3 $. We have $ v_j = 2 $, $ v_{j+1} = u_{i+3}+2 $, and $ v_{j-1} = 2 $. We get $ 2 \leq e \leq 3 $, $ 0 \leq f \leq u_{i+3} $, and $ e \neq 3 $, $ (e, f) \neq (2,u_{i+3}) $. This gives us (j).
\end{proof}

\begin{lemma}
\label{lem:D2b2}
Let $ \alpha = e\alpha_{i,r}+f\alpha_{i,r+1} $, where $ i \geq -1 $ is odd and $ 0 \leq r \leq u_{i+2}-1 $. The element $ \alpha $ satisfies condition (2) in \Cref{thm:D2} if and only if one of the following holds:
\begin{enumerate}[(a)]
\item[(a)] $ r = 0 $, $ 1 \leq e \leq u_{i+1}-1 $, $ f = 2 $, $ u_{i+1} \geq 2 $, and $ u_{i+2} \geq 3 $,
\item[(b)] $ r = 0 $, $ 1 \leq e \leq u_{i+1}-1 $, $ 2 \leq f \leq 3 $, $ u_{i+1} \geq 2 $, and $ u_{i+2} = 2 $,
\item[(c)] $ r = 0 $, $ e = u_{i+1} $, $ f = 2 $, and $ u_{i+2} = 2 $,
\item[(d)] $ r = 0 $, $ 1 \leq e \leq u_{i+1}-1 $, $ u_{i+3}+2 \leq f \leq 2u_{i+3}+2 $, $ u_{i+1} \geq 2 $, $ u_{i+2} = 1 $, and $ u_{i+4} \geq 2 $,
\item[(e)] $ r = 0 $, $ e = u_{i+1} $, $ u_{i+3}+2 \leq f \leq 2u_{i+3}+1 $, $ u_{i+2} = 1 $, and $ u_{i+4} \geq 2 $,
\item[(f)] $ r = 0 $, $ 1 \leq e \leq u_{i+1}-1 $, $ u_{i+3}+2 \leq f \leq 2u_{i+3}+3 $, $ u_{i+1} \geq 2 $, $ u_{i+2} = 1 $, and $ u_{i+4} = 1 $,
\item[(g)] $ r = 0 $, $ e = u_{i+1} $, $ u_{i+3}+2 \leq f \leq 2u_{i+3}+2 $, $ u_{i+2} = 1 $, and $ u_{i+4} = 1 $.
\end{enumerate}
\end{lemma}
\begin{proof}
Let $ j \in \Z_{\geq 0} $ be such that $ \beta_j = \alpha_{i,r} $, hence $ \alpha = e\beta_j+f\beta_{j+1} $. Assume that $ \alpha $ satisfies condition (2) in \Cref{thm:D2}, i.e., $ 1 \leq e \leq v_j-2 $, $ v_{j+1} \leq f \leq 2v_{j+1}-1 $, and
\[
	(e,f) \neq (v_j-2,2v_{j+1}-1),\quad (f,v_{j+2}) \neq (2v_{j+1}-1,2),\quad (e,f,v_{j+2}) \neq (v_j-2,2v_{j+1}-2,2).
\]
Case $ r = 0 $: we have $ v_j = u_{i+1}+2 $. If $ u_{i+2} \geq 2 $, then $ v_{j+1} = 2 $, hence $ 1 \leq e \leq u_{i+1} $ and $ 2 \leq f \leq 3 $. If $ u_{i+2} \geq 3 $, then $ v_{j+2} = 2 $, hence $ f \neq 3 $ and $ (e,f) \neq (u_{i+1},2) $. This gives us (a). If $ u_{i+2} = 2 $, then $ v_{j+2} = u_{i+3}+2 > 2 $, hence $ (e, f) \neq (u_{i+1},3) $. This gives us (b) and (c).

If $ u_{i+2} = 1 $, then $ v_{j+1} = u_{i+3}+2 $, hence $ 1 \leq e \leq u_{i+1} $ and $ u_{i+3}+2 \leq f \leq 2u_{i+3}+3 $. If $ u_{i+4} \geq 2 $, then $ v_{j+2} = 2 $. We get $ f \neq 2u_{i+3}+3 $ and $ (e, f) \neq (u_{i+1},2u_{i+3}+2) $. This gives us (d) and (e). If $ u_{i+4} = 1 $, then $ v_{j+2} = u_{i+5}+2 > 2 $. We get $ (e, f) \neq (u_{i+1},2u_{i+3}+3) $. This gives us (f) and (g).

Case $ 1 \leq r \leq u_{i+2}-1 $: we get $ v_j = 2 $, hence $ 1 \leq e \leq 0 $, a contradiction.
\end{proof}

\begin{lemma}
\label{lem:D2b3}
Let $ \alpha = e\alpha_{i,r}+f\alpha_{i,r+1} $, where $ i \geq -1 $ is odd and $ 0 \leq r \leq u_{i+2}-1 $. The element $ \alpha $ satisfies condition (3) in \Cref{thm:D2} if and only if one of the following holds:
\begin{enumerate}[(a)]
\item[(a)] $ r = 0 $, $ e = u_{i+1}+1 $, $ f = 1 $, and $ u_{i+2} \geq 2 $,
\item[(b)] $ r = 0 $, $ e = u_{i+1}+1 $, $ f = u_{i+3}+1 $, and $ u_{i+2} = 1 $,
\item[(c)] $ r = 1 $, $ e = 1 $, $ f = 1 $, and $ u_{i+2} \geq 3 $,
\item[(d)] $ r = 1 $, $ e = 1 $, $ f = u_{i+3}+1 $, and $ u_{i+2} = 2 $,
\item[(e)] $ r = u_{i+2}-2 $, $ e = 1 $, $ f = 1 $, and $ u_{i+2} \geq 4 $,
\item[(f)] $ r = u_{i+2}-1 $, $ e = 1 $, $ f = u_{i+3}+1 $, and $ u_{i+2} \geq 3 $.
\end{enumerate}
\end{lemma}
\begin{proof}
Let $ j \in \Z $ be such that $ \beta_j = \alpha_{i,r} $, hence $ \alpha = e\beta_j+f\beta_{j+1} $. Assume that $ \alpha $ satisfies condition (3) in \Cref{thm:D2}, i.e.,
\[
	e = v_j-1,\quad f = v_{j+1}-1,\quad\text{and}\quad(v_{j-1},v_j,v_{j+1},v_{j+2}) \neq (2,2,2,2).
\]

Case $ r = 0 $: we have $ v_j = u_{i+2}+2 $. If $ u_{i+2} \geq 2 $, then $ v_{j+1} = 2 $, which gives us (a). If $ u_{i+2} = 1 $, then $ v_{j+1} = u_{i+3}+2 $, which gives us (b).

Case $ r = 1 $: we must have $ u_{i+2} \geq 2 $ and $ v_j = 2 $. If $ u_{i+2} \geq 3 $, then $ v_{j+1} = 2 $. Moreover, $ v_{j-1} = u_{i+1}+2 > 2 $. This gives us (c). If $ u_{i+2} = 2 $, then $ v_{j+1} = u_{i+3}+2 $, and this gives us (d).

Case $ 2 \leq r \leq u_{i+2}-3 $: we have $ v_j = 2 $, $ v_{j+1} = 2 $, $ v_{j-1} = 2 $, and $ v_{j+1} = 2 $, a contradiction.

Case $ r = u_{i+2}-2 $: since we have already treated the cases $ r = 0 $ and $ r = 1 $, we may assume $ u_{i+2} \geq 4 $. We have $ v_j = 2 $, $ v_{j+1} = 2 $, and $ v_{j+2} = u_{i+3}+2 $, giving us (e).

Case $ r = u_{i+2}-1 $: we may assume $ u_{i+2} \geq 3 $. Now $ v_j = 2 $ and $ v_{j+1} = u_{i+3}+2 $, which gives us (f).
\end{proof}

\begin{theorem}
\label{thm:D2b}
All the elements $ \alpha \in \O_K^+ $ such that $ p_K(\alpha|\ind) = 2 $ are the following (where $ i \geq -1 $ is odd):
\begin{itemize}
\item $ \alpha = e\alpha_{i,0}+f\alpha_{i,1} $ with
\begin{enumerate}[(a)]
\item[(a)] $ u_{i+1}+2 \leq e \leq 2u_{i+1}+1 $ and $ f = 0 $ if $ u_{i+2} \geq 2 $,
\item[(b)] $ e = 2u_{i+1}+2 $ and $ f = 0 $ if $ u_{|i|} = 1 $ and $ u_{i+2} \geq 2 $,
\item[(c)] $ u_{i+2}+2 \leq e \leq 2u_{i+1}+1 $ and $ 0 \leq f \leq u_{i+3} $ if $ u_{|i|}\geq 2 $ and $ u_{i+2} = 1 $,
\item[(d)] $ e = 2u_{i+2}+2 $ and $ 0 \leq f \leq u_{i+3}-1 $ if $ u_{|i|} \geq 2 $ and $ u_{i+2} = 1 $,
\item[(e)] $ u_{i+1}+2 \leq e \leq 2u_{i+1}+2 $ and $ 0 \leq f \leq u_{i+3} $ if $ u_{|i|} = 1 $ and $ u_{i+2} = 1 $,
\item[(f)] $ e = 2u_{i+1}+3 $ and $ 0 \leq f \leq u_{i+3}-1 $ if $ u_{|i|} = 1 $ and $ u_{i+2} = 1 $,
\item[(g)] $ 1 \leq e \leq u_{i+1}-1 $ and $ f = 2 $ if $ u_{i+1} \geq 2 $ and $ u_{i+2} \geq 3 $,
\item[(h)] $ 1 \leq e \leq u_{i+1}-1 $ and $ 2 \leq f \leq 3 $ if $ u_{i+1} \geq 2 $ and $ u_{i+2} = 2 $,
\item[(i)] $ e = u_{i+1} $ and $ f = 2 $ if $ u_{i+2} = 2 $,
\item[(j)] $ 1 \leq e \leq u_{i+1}-1 $ and $ u_{i+3}+2 \leq f \leq 2u_{i+2}+2 $ if $ u_{i+1} \geq 2 $, $ u_{i+2} = 1 $, and $ u_{i+4} \geq 2 $,
\item[(k)] $ e = u_{i+1} $ and $ u_{i+3}+2 \leq f \leq 2u_{i+3}+1 $ if $ u_{i+2} = 1 $ and $ u_{i+4} \geq 2 $,
\item[(l)] $ 1 \leq e \leq u_{i+1}-1 $ and $ u_{i+3}+2 \leq f \leq 2u_{i+3}+3 $ if $ u_{i+1} \geq 2 $, $ u_{i+2} = 1 $, and $ u_{i+4} = 1 $,
\item[(m)] $ e = u_{i+1} $ and $ u_{i+3}+2 \leq f \leq 2u_{i+3}+2 $ if $ u_{i+2} = 1 $ and $ u_{i+4} = 1 $,
\item[(n)] $ e = u_{i+1}+1 $ and $ f = 1 $ if $ u_{i+2} \geq 2 $,
\item[(o)] $ e = u_{i+1}+1 $ and $ f = u_{i+3}+1 $ if $ u_{i+2} = 1 $.
\end{enumerate}
\item $ \alpha = e\alpha_{i,1}+f\alpha_{i,2} $ with
\begin{enumerate}[(a)]
\setcounter{enumi}{15}
\item[(p)] $ e = 2 $ and $ f = 0 $ if $ u_{i+2} \geq 3 $,
\item[(q)] $ e = 2 $ and $ 0 \leq f \leq u_{i+3} $ if $ u_{i+2} = 2 $,
\item[(r)] $ e = 3 $ and $ 0 \leq f \leq u_{i+3}-1 $ if $ u_{i+2} = 2 $,
\item[(s)] $ e = 1 $ and $ f = 1 $ if $ u_{i+2} \geq 3 $,
\item[(t)] $ e = 1 $ and $ f = u_{i+3}+1 $ if $ u_{i+2} = 2 $,
\end{enumerate}
\item $ \alpha = e\alpha_{i, u_{i+2}-2}+f\alpha_{i, u_{i+2}-1} $ with $ e = 1 $ and $ f = 1 $ if $ u_{i+2} \geq 4 $,
\item $ \alpha = e\alpha_{i, u_{i+2}-1}+f\alpha_{i+2,0} $ with
\begin{enumerate}[(a)]
\setcounter{enumi}{20}
\item[(u)] $ e = 2 $ and $ 0 \leq f \leq u_{i+3}-1 $ if $ u_{i+2} \geq 3 $,
\item[(v)] $ e = 1 $ and $ f = u_{i+3}+1 $ if $ u_{i+2} \geq 3 $,
\end{enumerate}
\item conjugates of all of the above.
\end{itemize}
\end{theorem}
\begin{proof}
This follows from \Cref{thm:D2} by putting together the conditions in \Cref{lem:D2b1,lem:D2b2,lem:D2b3}.
\end{proof}

\begin{theorem}
\label{thm:N2}
If $ \alpha \in \O_K^+ $ can be expressed as a sum of indecomposables in $2$ ways, i.e., $ p_K(\alpha|\ind) = 2 $, then
\[
	\Nm(\alpha) < 5\sqrt{\Delta}\left(\sqrt{\Delta}+1\right)\left(3\sqrt{\Delta}+2\right),
\]
where $ \Delta $ is the discriminant of $ K $.
\end{theorem}
\begin{proof}
We need to estimate the norms of the elements $ \alpha = e\alpha_{i,r}+f\alpha_{i,r+1} $ in \Cref{thm:D2b}.

First, suppose that $ r = 0 $. By \Cref{lem:Na3}, we have
\[
	\Nm(\alpha) < \sqrt{\Delta}(e+2f)(e+f).
\]
Now we substitute the bounds for $ e $ and $ f $ from cases (a)--(o) in \Cref{thm:D2b}. For example, in (a) we have $ e \leq 2u_{i+1}+1 < 2\sqrt{\Delta}+1 $ and $ f = 0 $, hence
\[
	\Nm(\alpha) < \sqrt{\Delta}\left(2\sqrt{\Delta}+1\right)^2.
\]
The worst case is (l), where $ e \leq u_{i+1}-1 < \sqrt{\Delta}-1 $ and $ f \leq 2u_{i+3}+3 < 2\sqrt{\Delta}+3 $, hence
\[
	\Nm(\alpha) < \sqrt{\Delta}\left(5\sqrt{\Delta}+5\right)\left(3\sqrt{\Delta}+2\right).
\]

Secondly, suppose that $ r = 1 $. By \Cref{lem:Na3}, we have
\[
	\Nm(\alpha) < \sqrt{\Delta}(2e+3f)(e+f).
\]
We analyze the cases (p)--(t) in \Cref{thm:D2b}. The worst case is (t), where $ e = 1 $ and $ f = u_{i+3}+1 < \sqrt{\Delta}+1 $, hence
\[
	\Nm(\alpha) < \sqrt{\Delta}\left(3\sqrt{\Delta}+5\right)\left(\sqrt{\Delta}+2\right).
\]

Next, suppose that $ r = u_{i+2}-2 $. By \Cref{lem:Na4}, we have
\[
	\Nm(\alpha) < \sqrt{\Delta}(4e+3f)(3e+2f).
\]
From \Cref{thm:D2b}, we get $ e = 1 $ and $ f = 1 $, hence $ \Nm(\alpha) < 35\sqrt{\Delta} $.

Finally, suppose that $ r = u_{i+2}-1 $. By \Cref{lem:Na4}, we have
\[
	\Nm(\alpha) < \sqrt{\Delta}(3e+2f)(2e+f).
\]
We look at the cases (u) and (v) in \Cref{thm:D2b}. In (u), we have $ e = 2 $ and $ f \leq u_{i+3}-1 < \sqrt{\Delta}-1 $, hence
\[
	\Nm(\alpha) < \sqrt{\Delta}\left(2\sqrt{\Delta}+4\right)\left(\sqrt{\Delta}+3\right).
\]
In (v), we have $ e = 1 $ and $ f = u_{i+3}+1 < \sqrt{\Delta}+1 $, hence
\[
	\Nm(\alpha) < \sqrt{\Delta}\left(2\sqrt{\Delta}+5\right)\left(\sqrt{\Delta}+3\right).
\]
This proves the bound for $ \Nm(\alpha) $ in each case.
\end{proof}

\section{Elements with a small number of partitions}
\label{sec:PK}

In this section, we use our results about partitions with indecomposable parts to describe all the elements with $ 6 $ partitions and determine $ \calD(6) $. A sufficient condition for the existence of $ \alpha \in \O_K^+ $ with $ 6 $ partitions was found in \cite{SZ}, where it was also remarked that this condition is not necessary.

\begin{theorem}[{\cite[Theorem 12]{SZ}}]
\label{thm:P6Suf}
Let $ K = \Q(\sqrt{D}) $, where $ D \in \Z_{\geq 2} $ is squarefree, $ D \neq 5 $, and let $ \alpha = \left(\lceil 2\xi_D \rceil+2\right)+2\omega_D $.
\begin{itemize}
\item If $ \lceil \xi_D \rceil-\xi_D > \frac{1}{2} $, then $ p_K(\alpha) = 6 $,
\item If $ \lceil \xi_D \rceil-\xi_D < \frac{1}{2} $, then $ p_K(\alpha) = 9 $.
\end{itemize}
\end{theorem}

Since $ \xi_D = \omega_D $ if $ D \equiv 2,3 \pmod{4} $ and $ \xi_D = \omega_D-1 $ if $ D \equiv 1 \pmod{4} $, we have that $ \lceil\xi_D\rceil-\xi_D > \frac{1}{2} $ is equivalent to $ \omega_D-\lfloor\omega_D\rfloor < \frac{1}{2} $. From the continued fraction expansion
\[
	\omega_D-\lfloor\omega_D\rfloor = \frac{1}{u_1+\frac{1}{u_2+\cdots}},
\]
we see that this is equivalent to $ u_1 \geq 2 $. We also note that $ D \neq 5 $ is equivalent to $ u_0 = 2\lfloor\omega_D\rfloor-\Tr(\omega_D) \geq 2 $. Thus, if $ u_0 \geq 2 $ and $ u_1 \geq 2 $, then the element $ \alpha $ from \Cref{thm:P6Suf} satisfies $ p_K(\alpha) = 6 $.

The set $ \calD(6) $ will be completely determined in terms of the continued fraction of $ \omega_D $ in \Cref{thm:E6}.

\begin{lemma}
\label{lem:PK6}
Let $ \alpha \in \O_K^{+} $ and $ j \in \Z $, $ e \in \Z_{\geq 1} $, $ f \in \Z_{\geq 0} $ such that $ \alpha = e\beta_j+f\beta_{j+1} $. If $ p_K(\alpha) \leq 6 $, then $ (e,f) \in \{(1,0),(2,0),(3,0),(4,0),(1,1),(2,1),(1,2)\} $.
\end{lemma}
\begin{proof}
We show that if $ \alpha \succeq 5\beta_j $ or $ \alpha \succeq 3\beta_j+\beta_{j+1} $ or  $ \alpha \succeq \beta_j+3\beta_{j+1} $ or $ \alpha \succeq 2\beta_j+2\beta_{j+1} $, then $ p_K(\alpha) \geq 7 $.

The element $ 5\beta_j $ has at least $7$ partitions corresponding to the $7$ partitions of $5$, namely
\begin{gather*}
	(5\beta_j),\ (4\beta_j, \beta_j),\ (3\beta_j, 2\beta_j),\ (3\beta_j, \beta_j, \beta_j),\ (2\beta_j, 2\beta_j, \beta_j),\\
	(2\beta_j, \beta_j, \beta_j, \beta_j),\ \text{and}\ (\beta_j, \beta_j, \beta_j, \beta_j, \beta_j).
\end{gather*}
The element $ 3\beta_j+\beta_{j+1} $ has at least the following $7$ partitions:
\begin{gather*}
	(3\beta_j+\beta_{j+1}),\ (3\beta_j, \beta_{j+1}),\ (2\beta_j, \beta_j+\beta_{j+1}),\ (\beta_j, 2\beta_j+\beta_{j+1}),\ (2\beta_j, \beta_j, \beta_{j+1}),\\
	(\beta_j, \beta_j, \beta_j+\beta_{j+1}),\ \text{and}\ (\beta_j, \beta_j, \beta_j, \beta_{j+1}).
\end{gather*}
Similarly, $ \beta_j+3\beta_{j+1} $ has at least $7$ partitions obtained from the partitions of $ 3\beta_j+\beta_{j+1} $ by exchanging the roles of $ \beta_j $ and $ \beta_{j+1} $.

The element $ 2\beta_j+2\beta_{j+1} $ also has at least $7$ partitions:
\begin{gather*}
	(2\beta_j+2\beta_{j+1}),\ (2\beta_j+\beta_{j+1},\beta_{j+1}),\ (2\beta_j,2\beta_{j+1}),\ (\beta_j, \beta_j+2\beta_{j+1}),\ (2\beta_j, \beta_{j+1}, \beta_{j+1}),\\
	(\beta_j, \beta_j, 2\beta_{j+1})\ \text{and}\ (\beta_j, \beta_j, \beta_{j+1}, \beta_{j+1}).\qedhere
\end{gather*}
\end{proof}

\begin{lemma}
\label{lem:P20}
Let $ i \geq -1 $ be odd and $ 0 \leq r \leq u_{i+2}-1 $. For $ \alpha = 2\alpha_{i,r} $, we have $ p_K(\alpha) = m $ if and only if either $ r = m-2 $ and $ u_{i+2} \geq 2m-4 $, or $ r = u_{i+2}-(m-2) $ and $ u_{i+2} \geq 2m-3 $.

For $ \alpha = \alpha_{i,r}+\alpha_{i,r+1} $, we have $ p_K(\alpha) = m $ if and only if either $ r = m-2 $ and $ u_{i+2} \geq 2m-3 $, or $ r = u_{i+2}-(m-1) $ and $ u_{i+2} \geq 2m-2 $.
\end{lemma}
\begin{proof}
By \Cref{prop:EF}, we have
\[
	p_K(2\alpha_{i,r}) = \min\{r+2,u_{i+2}-r+2\},
\]
hence $ p_K(2\alpha_{i,r}) = m $ if and only if either $ r+2 = m $ and $ u_{i+2}-r+2 \geq m $, or $ r+2 > m $ and $ u_{i+2}-r+2 = m $.

Similarly,
\[
	p_K(\alpha_{i,r}+\alpha_{i,r+1}) = \min\{r+2,u_{i+2}-r+1\},
\]
hence $ p_K(\alpha_{i,r}+\alpha_{i,r+1}) = m $ if and only if either $ r+2 = m $ and $ u_{i+2}-r+1 \geq m $, or $ r+2 > m $ and $ u_{i+2}-r+1 = m $.
\end{proof}

\begin{lemma}
\label{lem:P30}
If $ \alpha = 3\beta_j $, where $ j \in \Z $, then we have the following:
\begin{itemize}
\item if $ v_j \geq 4 $, then $ p_K(\alpha) = 3 $,
\item if $ v_j = 3 $, then $ p_K(\alpha) = 4 $,
\item if $ v_j = 2 $, $ v_{j-1} > 2 $, and $ v_{j+1} > 2 $, then $ p_K(\alpha) = 6 $,
\item if $ v_j = 2 $ and $ v_{j-1} = 2 $ or $ v_{j+1} = 2 $, then $ p_K(\alpha) \geq 8 $.
\end{itemize}
\end{lemma}
\begin{proof}
If $ v_j \geq 4 $, then $ p_K(\alpha|\ind) = 1 $ by \Cref{thm:D1}. The only partition of $ \alpha $ with indecomposable parts is $ (\beta_j,\beta_j,\beta_j) $, and all partitions of $ \alpha $ are
\begin{equation}
\label{eq:P30}
	(3\beta_j),\ (2\beta_j,\beta_j),\ \text{and}\ (\beta_j,\beta_j,\beta_j),
\end{equation}
hence $ p_K(\alpha) = 3 $.

If $ v_j = 3 $, then $ \alpha $ satisfies condition (1) in \Cref{thm:D2}, hence $ p_K(\alpha|\ind) = 2 $. We have $ 3\beta_j = \beta_{j-1}+\beta_{j+1} $, and all partitions of $ \alpha $ are the ones listed in \eqref{eq:P30} together with $ (\beta_{j-1},\beta_{j+1}) $, hence $ p_K(\alpha) = 4 $.

If $ v_j = 2 $, $ v_{j-1} > 2 $, and $ v_{j+1} > 2 $, then $ \alpha $ again satisfies condition (1) in \Cref{thm:D2}, hence $ p_K(\alpha|\ind) = 2 $. We have $ 3\beta_j = \beta_{j-1}+\beta_j+\beta_{j+1} $ and all the partitions of $ \alpha $ are the ones listed in \eqref{eq:P30} together with
\begin{equation}
\label{eq:P30b}
	(\beta_{j-1},\beta_j+\beta_{j+1}),\ (\beta_{j-1}+\beta_j,\beta_{j+1}),\ \text{and}\ (\beta_{j-1},\beta_j,\beta_{j+1}),
\end{equation}
hence $ p_K(\alpha) = 6 $.

If $ v_j = 2 $ and $ v_{j-1} = 2 $, then
\[
	3\beta_j = \beta_{j-1}+\beta_j+\beta_{j+1} = \beta_{j-2}+2\beta_{j+1}.
\]
The element $ \alpha $ has the $ 6 $ partitions in \eqref{eq:P30} and \eqref{eq:P30b}, together with
\[
	(\beta_{j-2},2\beta_{j+1})\ \text{and}\ (\beta_{j-2},\beta_{j+1},\beta_{j+1}),
\]
hence $ p_K(\alpha) \geq 8 $.

If $ v_j = 2 $ and $ v_{j+1} = 2 $, then
\[
	3\beta_j = \beta_{j-1}+\beta_j+\beta_{j+1} = 2\beta_{j-1}+\beta_{j+2},
\]
and we again have $ p_K(\alpha) \geq 8 $.
\end{proof}

\begin{lemma}
\label{lem:P40}
If $ \alpha = 4\beta_j $, where $ j \in \Z $, then we have the following:
\begin{itemize}
\item if $ v_j \geq 5 $, then $ p_K(\alpha) = 5 $,
\item if $ v_j = 4 $, then $ p_K(\alpha) = 6 $,
\item if $ v_j = 3 $, then $ p_K(\alpha) \geq 8 $,
\item if $ v_j = 2 $, then $ p_K(\alpha) \geq 16 $.
\end{itemize}
\end{lemma}
\begin{proof}
If $ v_j \geq 5 $, then $ p_K(\alpha|\ind) = 1 $ by \Cref{thm:D1}. It follows that $ \alpha $ has $ 5 $ partitions corresponding to the partitions of $ 4 $, namely
\begin{equation}
\label{eq:P40a}
	(4\beta_j),\ (3\beta_j,\beta_j),\ (2\beta_j,2\beta_j),\ (2\beta_j,\beta_j,\beta_j),\text{ and }(\beta_j,\beta_j,\beta_j,\beta_j).
\end{equation}

If $ v_j = 4 $, then $ \alpha $ satisfies condition (1) in \Cref{thm:D2}, hence $ p_K(\alpha|\ind) = 2 $. We have
\[
	4\beta_j = \beta_{j-1}+\beta_{j+1},
\]
and $ \alpha $ has the $ 5 $ partitions listed in \eqref{eq:P40a} together with $ (\beta_{j-1},\beta_{j+1}) $, hence $ p_K(\alpha) = 6 $.

If $ v_j = 3 $, then
\[
	4\beta_j = \beta_{j-1}+\beta_j+\beta_{j+1},
\]
and $ \alpha $ has the $ 5 $ partitions in \eqref{eq:P40a} together with
\[
	(\beta_{j-1},\beta_j+\beta_{j+1}),\ (\beta_{j-1}+\beta_j,\beta_{j+1}),\ \text{and}\ (\beta_{j-1},\beta_j,\beta_{j+1}),
\]
hence $ p_K(\alpha) \geq 8 $.

If $ v_j = 2 $, then
\[
	4\beta_j = \beta_{j-1}+2\beta_j+\beta_{j+1} = 2\beta_{j-1}+2\beta_{j+1},
\]
hence $ \alpha $ has the $ 5 $ partitions in \eqref{eq:P40a} together with the $ 7 $ partitions
\begin{gather*}
	(\beta_{j-1},2\beta_j+\beta_{j+1}),\ (\beta_{j-1}+2\beta_j,\beta_{j+1}),\ (\beta_{j-1}+\beta_j,\beta_j+\beta_{j+1}),\ (\beta_{j-1},2\beta_j,\beta_{j+1}),\\
	(\beta_{j-1}+\beta_j,\beta_j,\beta_{j+1}),\ (\beta_{j-1},\beta_j,\beta_j+\beta_{j+1}),\ (\beta_{j-1},\beta_j,\beta_j,\beta_{j+1}) 
\end{gather*}
and the $ 4 $ partitions
\[
	(2\beta_{j-1},2\beta_{j+1}),\ (2\beta_{j-1},\beta_{j+1},\beta_{j+1}),\ (\beta_{j-1},\beta_{j-1},2\beta_{j+1}),\ (\beta_{j-1},\beta_{j-1},\beta_{j+1},\beta_{j+1}),
\]
hence $ p_K(\alpha) \geq 16 $.
\end{proof}

\begin{lemma}
\label{lem:P21}
If $ \alpha = 2\beta_j+\beta_{j+1} $, where $ j \in \Z $, then we have the following:
\begin{itemize}
\item if $ v_j \geq 3 $ and $ (v_j,v_{j+1}) \neq (3,2) $, then $ p_K(\alpha) = 4 $,
\item if $ (v_j, v_{j+1}) = (3,2) $, then $ p_K(\alpha) = 5 $,
\item if $ v_j = 2 $ and $ v_{j+1} \geq 4 $, then $ p_K(\alpha) = 6 $,
\item if $ v_j = 2 $, $ v_{j+1} = 3 $, and $ v_{j-1} > 2 $, then $ p_K(\alpha) = 6 $,
\item if $ v_j = 2 $, $ v_{j+1} = 3 $, and $ v_{j-1} = 2 $, then $ p_K(\alpha) = 7 $,
\item if $ v_j = 2 $ and $ v_{j+1} = 2 $, then $ p_K(\alpha) \geq 8 $.
\end{itemize}
\end{lemma}
\begin{proof}
If $ v_j \geq 3 $ and $ (v_j, v_{j+1}) \neq (3,2) $, then $ p_K(\alpha,\ind) = 1 $ by \Cref{thm:D1}, and all partitions of $ \alpha $ are
\begin{equation}
\label{eq:P21a}
	(2\beta_j+\beta_{j+1}),\ (\beta_j,\beta_j+\beta_{j+1}),\ (2\beta_j,\beta_{j+1}),\text{ and }(\beta_j,\beta_j,\beta_{j+1}).
\end{equation}

If $ (v_j, v_{j+1}) = (3,2) $, then $ \alpha $ satisfies condition (3) in \Cref{thm:D2}, hence $ p_K(\alpha,\ind) = 2 $. We have
\[
	2\beta_j+\beta_{j+1} = \beta_{j-1}+\beta_{j+2},
\]
and all the partitions of $ \alpha $ are the $ 4 $ partitions listed in \eqref{eq:P21a} together with $ (\beta_{j-1},\beta_{j+2}) $.

If $ v_j = 2 $ and $ v_{j+1} \geq 4 $, then $ \alpha $ satisfies condition (1) in \Cref{thm:D2}, hence $ p_K(\alpha,\ind) = 2 $. We have
\[
	2\beta_j+\beta_{j+1} = \beta_{j-1}+2\beta_{j+1}
\]
and $ \alpha $ has the $ 4 $ partitions in \eqref{eq:P21a} together with
\begin{equation}
\label{eq:P21b}
	(\beta_{j-1},2\beta_{j+1})\text{ and }(\beta_{j-1},\beta_{j+1},\beta_{j+1}),
\end{equation}
hence $ p_K(\alpha) = 6 $.

If $ v_j = 2 $, $ v_{j+1} = 3 $, and $ v_{j-1} > 2 $, then $ \alpha $ also satisfies condition (1) of \Cref{thm:D2}, and we have $ p_K(\alpha) = 6 $ as above.

If $ v_j = 2 $, $ v_{j+1} = 3 $, and $ v_{j-1} = 2 $, then
\[
	2\beta_j+\beta_{j+1} = \beta_{j-1}+2\beta_{j+1} = \beta_{j-2}+\beta_{j+2}.
\]
By \Cref{lem:V2}, if $ \beta_{j_3} \preceq \beta_{j-2}+\beta_{j+2} $ for some $ j_3 \in \Z $, $ j_3 > j+2 $ or $ j_3 < j-2 $, then $ v_k = 2 $ for $ k \in \{j-2,\dots,j+2\} $. But we have $ v_{j+1} = 3 $, so this does not occur. It follows that $ p_K(\alpha|\ind) = 3 $ and the partitions of $ \alpha $ are the ones in \eqref{eq:P21a} and \eqref{eq:P21b} together with $ (\beta_{j-2},\beta_{j+2}) $, hence $ p_K(\alpha) = 7 $.

If $ v_j = 2 $ and $ v_{j+1} = 2 $, then
\[
	2\beta_j+\beta_{j+1} = \beta_{j-1}+2\beta_{j+1} = 2\beta_{j-1}+\beta_{j+2}
\]
and $ \alpha $ has the $ 6 $ partitions in \eqref{eq:P21a} and \eqref{eq:P21b} together with
\[
	(2\beta_{j-1},\beta_{j+2})\text{ and }(\beta_{j-1},\beta_{j-1},\beta_{j+2}),
\]
hence $ p_K(\alpha) \geq 8 $.
\end{proof}

\begin{lemma}
\label{lem:P12}
If $ \alpha = \beta_j+2\beta_{j+1} $, where $ j \in \Z $, then we have the following:
\begin{itemize}
\item if $ (v_j,v_{j+1}) \neq (2,3) $ and $ v_{j+1} \geq 3 $, then $ p_K(\alpha) = 4 $,
\item if $ (v_j, v_{j+1}) = (2,3) $, then $ p_K(\alpha) = 5 $,
\item if $ v_j \geq 4 $ and $ v_{j+1} = 2 $, then $ p_K(\alpha) = 6 $,
\item if $ v_j = 3 $, $ v_{j+1} = 2 $, and $ v_{j+2} > 2 $, then $ p_K(\alpha) = 6 $,
\item if $ v_j = 3 $, $ v_{j+1} = 2 $, and $ v_{j+2} = 2 $, then $ p_K(\alpha) = 7 $,
\item if $ v_j = 2 $ and $ v_{j+1} = 2 $, then $ p_K(\alpha) \geq 8 $.
\end{itemize}
\end{lemma}
\begin{proof}
Let $ j' = -(j+1) $ and consider
\[
	\alpha' = \beta_j'+2\beta_{j+1}' = \beta_{-j}+2\beta_{-(j+1)} = 2\beta_{j'}+\beta_{j'+1}.
\]
We have $ v_{j'} = v_{j+1} $, $ v_{j'+1} = v_j $, and $ v_{j'-1} = v_{j+2} $. The lemma follows from \Cref{lem:P21} applied to $ \alpha' $.
\end{proof}

With the help of the preceding series of lemmas, it is possible to characterize the elements $ \alpha \in \O_K^+ $ such that $ p_K(\alpha) = m $ for $ m \in \{2,3,4,5,6\} $. We do this for $ m = 6 $ and then find a necessary and sufficient condition for $ K $ to contain an element with $ 6 $ partitions.

\begin{proposition}
\label{prop:P6}
All the elements $ \alpha \in \O_K^+ $ such that $ p_K(\alpha) = 6 $ are the following (where $ i \geq -1 $ is~odd):
\begin{enumerate}[(a)]
\item[(a)] $ \alpha = 2\alpha_{i,4} $ if $ u_{i+2} \geq 8 $,
\item[(b)] $ \alpha = 2\alpha_{i,u_{i+2}-4} $ if $ u_{i+2} \geq 9 $,
\item[(c)] $ \alpha = \alpha_{i,4}+\alpha_{i,5} $ if $ u_{i+2} \geq 9 $,
\item[(d)] $ \alpha = \alpha_{i,u_{i+2}-5}+\alpha_{i,u_{i+2}-4} $ if $ u_{i+2} \geq 10 $,
\item[(e)] $ \alpha = 3\alpha_{i,1} $ if $ u_{i+2} = 2 $,
\item[(f)] $ \alpha = 4\alpha_{i,0} $ if $ u_{i+1} = 2 $,
\item[(g)] $ \alpha = 2\alpha_{i,u_{i+2}-1}+\alpha_{i+2,0} $ if $ u_{i+2} \geq 2 $ and $ u_{i+3} \geq 2 $, or $ u_{i+2} = 2 $ and $ u_{i+3} = 1 $,
\item[(h)] $ \alpha = \alpha_{i,0}+2\alpha_{i,1} $ if $ u_{i+1} \geq 2 $ and $ u_{i+2} \geq 2 $, or $ u_{i+1} = 1 $ and $ u_{i+2} = 2 $,
\item[(i)] conjugates of all of the above.
\end{enumerate}
\end{proposition}
\begin{proof}
Let $ \alpha = e\alpha_{i,r}+f\alpha_{i,r+1} $, where $ i \geq -1 $ is odd and $ 0 \leq r \leq u_{i+2}-1 $. Let $ j \in \Z $ be such that $ \beta_j = \alpha_{i,r} $, so that $ \alpha = e\beta_j+f\beta_{j+1} $. Assume that $ p_K(\alpha)=6 $, hence
\[
	(e,f) \in \{(2,0),(3,0),(4,0),(1,1),(2,1),(1,2)\}
\]
by \Cref{lem:PK6}.

If $ (e,f) = (2,0) $, then \Cref{lem:P20} shows that $ r = 4 $ and $ u_{i+2} \geq 8 $, or $ r = u_{i+2}-4 $ and $ u_{i+2} \geq 9 $.

If $ (e,f) = (1,1) $, then \Cref{lem:P20} shows that $ r = 4 $ and $ u_{i+2} \geq 9 $, or $ r = u_{i+2}-5 $ and $ u_{i+2} \geq 10 $.

If $ (e,f) = (3,0) $, then \Cref{lem:P30} shows that $ v_j = 2 $, $ v_{j-1} > 2 $, and $ v_{j+1} > 2 $. Thus, $ r = 1 $ and $ u_{i+2} = 2 $.

If $ (e,f) = (4,0) $, then \Cref{lem:P40} shows that $ v_j = 4 $, thus $ r = 0 $ and $ u_{i+1} = 2 $.

If $ (e,f) = (2,1) $, then \Cref{lem:P21} shows that $ p_K(\alpha) = 6 $ if and only if one of the following holds:
\begin{itemize}
\item $ v_j = 2 $ and $ v_{j+1} \geq 4 $,
\item $ v_j = 2 $, $ v_{j+1} = 3 $, and $ v_{j-1} > 2 $.
\end{itemize}
These conditions translate to
\begin{itemize}
\item $ r = u_{i+2}-1 $, $ u_{i+2} \geq 2 $, and $ u_{i+3} \geq 2 $,
\item $ r = u_{i+2}-1 $, $ u_{i+2} = 2 $, and $ u_{i+3} = 1 $.
\end{itemize}

If $ (e,f) = (1,2) $, then the condition follows similarly from \Cref{lem:P12}.
\end{proof}

\begin{theorem}
\label{thm:E6}
Let $ K = \Q(\sqrt{D}) $, where $ D \in \Z_{\geq 2} $ is squarefree. Let $ \omega_D $ have the continued fraction expansion $ \omega_D = [\lceil u_0/2 \rceil, \overline{u_1, \dots, u_s}] $, where $ u_0 = u_s $. There exists $ \alpha \in \O_K^+ $ such that $ p_K(\alpha) = 6 $ if and only if at least one of the following conditions is satisfied:
\begin{itemize}
\item $ u_i \geq 8 $ for some $ i \geq 1 $ odd,
\item $ u_i = 2 $ for some $ i \geq 0 $,
\item $ u_i \geq 2 $ and $ u_{i+1} \geq 2 $ for some $ i \geq 0 $.
\end{itemize}
\end{theorem}
\begin{proof}
This immediately follows from \Cref{prop:P6}.
\end{proof}

We know from \Cref{thm:S2} that $ 6 \in p_K(\O_K^+) $ for ``almost all'' squarefree $ D \in \Z_{\geq 2} $. By \Cref{thm:E6}, a squarefree $ D \in \Z_{\geq 2} $ belongs to the set $ \calD(6) $ if and only if
\begin{itemize}
\item $ u_i \leq 7 $ for every $ i \geq 1 $ odd,
\item $ u_i \neq 2 $ for every $ i \geq 0 $,
\item if $ u_i \geq 2 $, then $ u_{i+1} = 1 $ for every $ i \geq 0 $.
\end{itemize}

\begin{example}
\label{ex:E6}
There exists an $ \alpha \in \O_K^+ $ such that $ p_K(\alpha) = 6 $ if and only if
\[
	D \in \{2,3,6,10,11,13,14,19,22,26,29,30,31,33,38,39,41,42,46,\dots\},
\]
and there does \emph{not} exist an $ \alpha \in \O_K^+ $ such that $ p_K(\alpha) = 6 $ if and only if
\[
	D \in \{5,7,15,17,21,23,34,35,37,43,47,\dots\}.
\]
\end{example}

\Cref{thm:E6} completes the description of $ \calD(m) $ for $ 1 \leq m \leq 7 $. There does not seem to be any fundamental reason why the same techniques could not be extended to determine all the elements with $ m $ partitions, and hence $ \calD(m) $ for other values of $ m $.


\begin{thebibliography}{99}
\bibitem{And76}
G. E. Andrews, \textit{The theory of partitions}, Encyclopedia Math. Appl., Vol. 2, Adisson-Wesley Publishing Co., Reading, Mass.-London-Amsterdam, 1976. xiv+255 pp.
\bibitem{AE}
G. E. Andrews and K. Eriksson, \textit{Integer partitions}, Cambridge University Press, Cambridge, 2004. x+141 pp.
\bibitem{BK1}
V. Blomer and V. Kala, \textit{Number fields without $ n $-ary universal quadratic forms}, Math. Proc. Cambridge Philos. Soc. \textbf{159} (2015), no. 2, 239--252.
\bibitem{BK2}
V. Blomer and V. Kala, \textit{On the rank of universal quadratic forms over real quadratic fields}, Doc. Math. \textbf{23} (2018), 15--34.
\bibitem{Bru83}
H. Brunotte, \textit{{Z}ur {Z}erlegung totalpositiver {Z}ahlen in {O}rdnungen totalreeller algebraischer {Z}ahlk\"{o}rper.(German)[On the decompositions of totally positive numbers in orders of totally real algebraic number fields]}, Arch. Math. (Basel) \textbf{41} (1983), no. 6, 502--503.
\bibitem{CL}
M. \v{C}ech, D. Lachman, J. Svoboda, M. Tinkov\'{a}, and K. Zemkov\'{a}, \textit{Universal quadratic forms and indecomposables over biquadratic fields}, Math. Nachr. \textbf{292} (2019), no. 3, 540--555.
\bibitem{DS}
A. Dress and R. Scharlau, \textit{Indecomposable totally positive numbers in real quadratic orders}, \textit{J. Number Theory} \textbf{14} (1982), no. 3, 292--306.
\bibitem{Du24}
A. Dubickas, \textit{Representations of a number in an arbitrary base with unbounded digits}, Georgian Math. J., \href{https://doi.org/10.1515/gmj-2023-2118}{https://doi.org/10.1515/gmj-2023-2118}, 2024.
\bibitem{FS}
L. Fukshansky and Y. Shi, \textit{Positive semigroups and generalized Frobenius numbers over totally real number fields}, Mosc. J. Comb. Number Theory \textbf{9} (2020), no. 1, 29--41.
\bibitem{HR}
G. H. Hardy and S. Ramanujan, \textit{Asymptotic formul{\ae} in combinatory analysis}, Proc. London Math. Soc. (2) \textbf{17} (1918), 75--115.
\bibitem{HK}
T. Hejda and V. Kala, \textit{Additive structure of totally positive quadratic integers}, Manuscripta Math. \textbf{163} (2020), no. 1-2, 263--278.
\bibitem{Jan}
S. W. Jang, \textit{The partitions of totally positive integers over real quadratic fields}, unpublished manuscript.
\bibitem{JK}
S. W. Jang and B. M. Kim, \textit{A refinement of the Dress-Scharlau theorem}, J. Number Theory \textbf{158} (2016), 234--243.
\bibitem{JKK1}
S. W. Jang, B. M. Kim, and K. H. Kim, \textit{The Euler-Glaisher theorem over totally real number fields}, \href{https://arxiv.org/abs/2311.18515}{arxiv:2311.18515}, Tokyo J. Math. (2025), to appear.
\bibitem{JKK2}
S. W. Jang, B. M. Kim, and K. H. Kim, \textit{The Sylvester theorem and the Rogers-Ramanujan identities over totally real number fields}, \href{https://arxiv.org/abs/2311.18514}{arxiv:2311.18514}, 2023.
\bibitem{KT}
V. Kala and M. Tinková, \textit{Universal quadratic forms, small norms, and traces in families of number fields}, Int. Math. Res. Not. IMRN (2023), no. 9, 7541--7577.
\bibitem{KY}
V. Kala and P. Yatsyna, \textit{On Kitaoka's conjecture and lifting problem for universal quadratic forms}, Bull. Lond. Math. Soc. \textbf{55} (2023), no. 2, 854--864.
\bibitem{KYZ}
V. Kala, P. Yatsyna, and B. \.{Z}mija, \textit{Real quadratic fields with a universal form of a given rank have density zero}, \href{https://arxiv.org/abs/2302.12080}{arxiv:2302.12080}, American J. Math. (2025), to appear. 
\bibitem{KZ}
V. Kala and M. Zindulka, \textit{Partitions into powers of an algebraic number}, Ramanujan J. \textbf{64} (2024), no. 2, 537--551.
\bibitem{Kim00}
B. M. Kim, \textit{Universal octonary diagonal forms over some real quadratic fields}, Comment. Math. Helv. \textbf{75} (2000), no. 3, 410--414.
\bibitem{Man24}
S. H. Man, \textit{Minimal rank of universal lattices and number of indecomposable elements in real multiquadratic fields}, Adv. Math. \textbf{447} (2024), Paper No. 109694, 38 pp.
\bibitem{Mei53}
G. Meinardus, \textit{\"{U}ber das {P}artitionenproblem eines reell-quadratischen {Z}ahlk\"{o}rpers} (German), Math. Ann. \textbf{126} (1953), 343--361.
\bibitem{Mei56}
G. Meinardus, \textit{{Z}ur additiven {Z}ahlentheorie in mehreren {D}imensionen. I} (German), Math. Ann. \textbf{132} (1956), 333--346.
\bibitem{Mit78}
T. Mitsui, \textit{On the partition problem in an algebraic number field}, Tokyo J. Math. \textbf{1} (1978), no. 2, 189--236.
\bibitem{Ono00}
K. Ono, \textit{Distribution of the partition function modulo $ m $}, Ann. of Math. (2) \textbf{151} (2000), no. 1, 293--307.
\bibitem{Rad37}
H. Rademacher, \textit{On the partition function $ p(n) $}, Proc. London Math. Soc. (2) \textbf{43} (1937), no. 4, 241--254.
\bibitem{Rad50}
H. Rademacher, \textit{Additive algebraic number theory}, Proceedings of the International Congress of Mathematicians, Cambridge, Mass., 1950, vol. 1, pp. 356--362, Amer. Math. Soc, Providence, RI, 1952.
\bibitem{Ram19}
S. Ramanujan, \textit{Some properties of $ p(n) $, the number of partitions of $ n $}, Proc. Cambridge Philos. Soc. \textbf{19} (1919), 207--210.
\bibitem{Ram21}
S. Ramanujan, \textit{Congruence properties of partitions}, Math. Z. \textbf{9} (1921), no. 1-2, 147--153.
\bibitem{Sie45}
C. L. Siegel, \textit{Sums of $m$th powers of algebraic integers}, Ann. of Math. (2) \textbf{46} (1945), 313--339.
\bibitem{SZ}
D. Stern and M. Zindulka, \textit{Partitions in real quadratic fields}, Math. Nachr. \textbf{298} (2025), 548--566.
\bibitem{TV}
M. Tinková and P. Voutier, \textit{Indecomposable integers in real quadratic fields}, J. Number Theory \textbf{212} (2020), 458--482.
\end{thebibliography}
\end{document}